\documentclass[11pt,a4paper,reqno]{amsproc}

\makeatletter
\usepackage{datetime}
\usepackage[pdfborder={0 0 0}]{hyperref}
\usepackage{upref,cases}
\hypersetup{citebordercolor=.1 .6 1}

\usepackage[T1]{fontenc}
\usepackage[utf8]{inputenc}

\usepackage[final]{graphicx}
\usepackage{geometry}
\IfFileExists{mymtpro2.sty}{\usepackage[subscriptcorrection]{mymtpro2}

  \def\cup{\cupprod}
  \def\cap{\capprod}
  \def\bigcup{\bigcupprod}
  
  \def\bigcupdisjoint{\mathop{\kern10pt\raisebox{4pt}{$\cdot$}\kern-12pt\bigcup}\limits}
}%
	{\usepackage{amssymb,bm,dsfont}

}

\renewcommand{\Im}{\ensuremath{\operatorname{Im}}}

\renewcommand{\emptyset}{\ensuremath{\varnothing}}

\DeclareMathOperator{\Ran}{ran}

\DeclareMathOperator{\Tr}{Tr}

\DeclareMathOperator{\idd}{id}

\DeclareMathOperator{\distt}{d}
\DeclareMathOperator{\intt}{int}

\renewcommand{\vec}[1]{\ensuremath{\boldsymbol{#1}}}

\newcommand{\id}{\ensuremath{\mathds{1}}}

\DeclareMathOperator{\esssup}{ess\,sup}
\DeclareMathOperator{\supp}{supp}
\DeclareMathOperator{\dist}{dist}

\numberwithin{equation}{section}
\newtheoremstyle{ttheorem}%
       {1.8ex\@plus1ex}                % space above
       {2.1ex\@plus1ex\@minus.5ex}      % space below
       {\itshape}           % body font
       {0pt}                   % indent amount
       {\bfseries}          % Theoremhead font
       {.}                  % Punctuation after theorem head
       {.5em}               % Space after theorem head
       {}                % Theorem head spec (can be left empty: `normal')

\newtheoremstyle{ddefinition}%
       {1.8ex\@plus1ex}                % space above
       {2.1ex\@plus1ex\@minus.5ex}      % space belowi
       {}           % body font
       {0pt}                   % indent amount
       {\bfseries}           % Theoremhead font
       {.}                  % Punctuation after theorem head
       {.5em}               % Space after theorem head
       {}                % Theorem head spec (can be left empty: `normal')

\newtheoremstyle{rremark}%
       {1.8ex\@plus1ex}                % space above
       {2.1ex\@plus1ex\@minus.5ex}      % space belowi
       {\normalfont}        % body font
       {0pt}                   % indent amount
       {\bfseries}           % Theoremhead font
       {.}                  % Punctuation after theorem head
       {.5em}               % Space after theorem head
       {}                   % Theorem head spec (can be left empty: `normal')

\theoremstyle{ttheorem}
\newtheorem{theorem}{Theorem}[section]
\newtheorem{lemma}[theorem]{Lemma}

\theoremstyle{ddefinition}

\theoremstyle{rremark}
\newtheorem{remark}[theorem]{Remark}
\newtheorem{myremarks}[theorem]{Remarks}
\newtheorem{myexamples}[theorem]{Examples}

\newenvironment{remarks}{\begin{myremarks}\begin{nummer}}%
    {\end{nummer}\end{myremarks}}
    {\end{nummer}\end{myexamples}}

\newcounter{numcount}
\newcommand{\labelnummer}{(\roman{numcount})}%

\makeatletter
\providecommand{\showkeyslabelformat}[1]{\relax}        %	 for compatibility with package showkeys
\let\mysaveformat\showkeyslabelformat                   %
\def\myformat#1{\raisebox{-1.5ex}{\mysaveformat{#1}}}   %

\newenvironment{nummer}%
  {\let\curlabelspeicher\@currentlabel%
    \begin{list}{\textup{\labelnummer}}%
      {\usecounter{numcount}\leftmargin0pt%
        \topsep0.5ex\partopsep2ex\parsep0pt\itemsep0ex\@plus1\p@%
        \labelwidth2.5em\itemindent3.5em\labelsep1em%
      }%
    \let\saveitem\item%
    \def\item{\saveitem%
      \def\@currentlabel{\curlabelspeicher\kern.1em\labelnummer}}%
    \let\savelabel\label%
    \def\label##1{{\ifnum\thenumcount=1\let\showkeyslabelformat\myformat\fi\savelabel{##1}}%
										{\def\@currentlabel{\labelnummer}%
									 	\let\showkeyslabelformat\@gobble%	 for compatibility with package showkeys
									 	\savelabel{##1item}%
										}%
	   							}%
  }{\end{list}}%

  {\let\curlabelspeicher\@currentlabel%
    \begin{list}{\textup{\labelnummer}}%
      {\usecounter{numcount}\leftmargin0pt%
        \topsep0.5ex\partopsep2ex\parsep0pt\itemsep0ex\@plus1\p@%
        \labelwidth2.5em\itemindent0em\labelsep1em%
        \leftmargin2.5em}%
    \let\saveitem\item%
    \def\item{\saveitem%
      \def\@currentlabel{\curlabelspeicher\kern.1em\labelnummer}}%
    \let\savelabel\label%
    \def\label##1{{\ifnum\thenumcount=1\let\showkeyslabelformat\myformat\fi\savelabel{##1}}%
										{\def\@currentlabel{\labelnummer}%
									 	\let\showkeyslabelformat\@gobble%	 	 for compatibility with package showkeys
									 	\savelabel{##1item}%
										}%
    							}%
  }{\end{list}}%

\def\section{\@startsection{section}{1}%
  \z@{1.3\linespacing\@plus\linespacing}{.5\linespacing}%
  {\normalfont\bfseries\centering}}

\def\subsection{\@startsection{subsection}{2}%
  \z@{.8\linespacing\@plus.5\linespacing}{-1em}%
  {\normalfont\bfseries}}

\def\nlsubsection{\@startsection{subsection}{2}%
  \z@{.8\linespacing\@plus.5\linespacing}{.1ex}%
  {\normalfont\bfseries}}

\let\@afterindenttrue\@afterindentfalse%

\renewenvironment{proof}[1][\proofname]{\par \normalfont
  \topsep6\p@\@plus6\p@ \trivlist %\itemindent\normalparindent
  \item[\hskip\labelsep\scshape
    #1\@addpunct{.}]\ignorespaces
}{%
  \qed\endtrivlist
}

\def\ps@firstpage{\ps@plain
  \def\@oddfoot{\normalfont\scriptsize \hfil\thepage\hfil
     \global\topskip\normaltopskip}%
  \let\@evenfoot\@oddfoot
  \def\@oddhead{
    \begin{minipage}{\textwidth}
      \normalfont\scriptsize
      \emph{\insertfirsthead}
    \end{minipage}}
  \let\@evenhead\@oddhead % in case an article starts on a left-hand page
}

\def\insertfirsthead{}

\def\@cite#1#2{{%
 \m@th\upshape\mdseries[{#1}{\if@tempswa, #2\fi}]}}
\newcommand{\C}{\mathbb{C}}
\newcommand{\N}{\mathbb{N}}

\newcommand{\R}{\mathbb{R}}
\newcommand{\Z}{\mathbb{Z}}
\let\<\langle
\let\>\rangle

\newcommand{\hairspace}{\kern .04167em}
\newcommand{\Zd}{\mathbb{Z}^d}

\DeclareMathOperator{\TextIm}{Im}

\renewcommand{\Im}{\TextIm}

\def\clap#1{\hbox to 0pt{\hss#1\hss}}

\makeatother
\usepackage{enumitem}
\setitemize{leftmargin=0pt}

\newcommand{\beq}{\begin{equation}}
\newcommand{\eeq}{\end{equation}}
\newcommand{\be}{\begin}
\newcommand{\e}{\end}

\newcommand{\E}{\mathbb E}

\usepackage{enumitem}

\let\OldItem\item% remember the previous definition
\newcommand{\MyItem}[2][]{}%
\begin{document}

\title[Trace asymptotics for ergodic operators]{Full Szeg\H{o}-type trace asymptotics for ergodic \\ operators on large boxes}

\author[A.\ Dietlein]{Adrian Dietlein}
\address[A.\ Dietlein]{Mathematisches Institut,
  Ludwig-Maximilians-Universit\"at M\"unchen,
  Theresienstra\ss{e} 39,
  80333 M\"unchen, Germany}
\email{dietlein@math.lmu.de}

\begin{abstract}
We prove full Szeg\H{o}-type large-box trace asymptotics for selfadjoint $\Z^d$-ergodic operators $\Omega\ni \omega\mapsto H_\omega$ acting on $L^2(\R^d)$.
More precisely, let $g$ be a bounded, compactly supported and real-valued function such that the (averaged) operator kernel of $g(H_\omega)$ decays sufficiently fast, and let $h$ be a sufficiently smooth compactly supported function. We then prove a full asymptotic expansion of the averaged trace $\Tr\left( h(g(H_\omega)_{[-L,L]^d}) \right)$ in terms of the length-scale $L$.
\end{abstract}

\maketitle

\section{Introduction}

Let $0<a:\mathbb{T}\to \R$ be a continuous symbol on the torus with Fourier coefficients $(a_k)_{k\in\Z}$. Szeg\H{o} observed in 1915 \cite{Szego1915} that
\begin{align}
\label{intro:SzegoDetWeak}
\log \det\left( (a_{j-k})_{j,k=1}^{L} \right) = L (\log a)_0 + \hbox{o}(L) \qquad (L\to\infty).  
\end{align}
Under the additional assumption that the symbol is, for simplicity, $\mathcal{C}^2$, he later found \cite{SzegoStrong} the two-term asymptotic expansion
\begin{align}
\label{intro:SzegoDet}
\log \det\left( (a_{j-k})_{j,k=1}^{L} \right) = L (\log a)_0 + \sum_{l=1}^\infty l (\log a)_l (\log a)_{-l} + \hbox{o}(1) \qquad (L\to\infty).  
\end{align}
The matrix on the left-hand side is the finite-volume truncation $A_L$ of the Toeplitz matrix $A = (a_{j-k})_{j,k\in \N}$. For sufficiently smooth test functions $h$, \eqref{intro:SzegoDet} implies the asymptotic trace formula
\begin{equation}
\label{intro:SzegoTr}
\Tr\left( h(A_L) \right) = L (h\circ a)_0 + A_1 + \hbox{o}(1)\qquad (L\to\infty),
\end{equation}
where $A_1$ depends on $a$ and $h$. The subleading term in \eqref{intro:SzegoTr} crucially depends on the smoothness of the symbol \cite{FiHaConj,BASOR198625}. For symbols $a$ with jump discontinuities the subleading term is typically of order $\log (L)$,
\begin{equation}
\label{intro:logenh}
\Tr\left( h(A_L) \right) = L (h\circ a)_0 + \widetilde{A}_1 \log L + \hbox{o}(\log L)\qquad (L\to\infty),
\end{equation}
where $\widetilde{A}_1$ again depends on $h$ and $a$ but on the latter only through the one-sided limits at its jump discontinuities. For further discussion of asymptotic expansions for determinants and traces of Toeplitz matrices we refer to
 \cite{BotSilLargeTruncToepl, KrasovskyToepDet, DeiftItsKrasToepDet}. We focus on the multi-dimensional continuum version of the problem. For a symbol $a:\R^d\to \C$ and a domain $\Omega\subset\R^d$, with $\Omega_L = L\Omega$, the truncated Wiener-Hopf operator $W_L(a):= \chi_{\Omega_L}\mathcal{F}^* a  \mathcal{F} \chi_{\Omega_L}$ is the multi-dimensional continuum analog of the truncated Toeplitz matrix $A_L$ from above. Here $\mathcal{F}$ denotes the Fourier transform and $\chi_{\Omega_L}$ is the spatial projection onto $\Omega_L$. If we assume, for example, that the domain $\Omega$ is piecewise smooth and the symbol $a$ is smooth and sufficiently fast decaying at infinity, then a natural analog of the asymptotic formula \eqref{intro:SzegoTr} holds for $W_L(a)$ and sufficiently smooth test functions $h$ with $h(0)=0$. The leading term now is of order $L^d$ and the subleading term is of order $L^{d-1}$, with an error term of order $\hbox{o}(L^{d-1})$. As in the one-dimensional Toeplitz case the subleading term depends on the smoothness of the symbol $a$. Again, an additional term of order $L^{d-1}\log L$ emerges if the symbol possesses jump-type discontinuities \cite{LanWidAsympt,Widom1982,HelLeschSpitAsympt,Sobolev2010,SobBook2013}. Motivated by its connection with the bipartite entanglement entropy those results have recently been extended to non-smooth test functions $h$, see \cite{PhysRevLett.112.160403, 1751-8121-49-30-30LT04} and references therein.
If the symbol $a$ is smooth and the domain $\Omega$ is not only piecewise smooth but smooth, then one can go beyond the subleading term \cite{MR752492,WidomTraceAsympt1985}. For $M\in \N$ 
\begin{align}
\label{intro:WidFullAsymp}
\Tr\left( h(W_L(a))) \right) = \sum_{m=0}^M L^{d-m} A_m + \hbox{o}(L^{d-M}) \qquad (L\to\infty)
\end{align}
holds for recursively defined coefficients $A_m=A_m(a,h,\Omega)$.

Recently, subleading-order trace asymptotics as in \eqref{intro:SzegoTr}--\eqref{intro:WidFullAsymp} have been studied for Schr\"odinger operators with non-trivial potential \cite{PasturSlavinAreaLawScaling,PasturKirschErgodic,PasElgSchLargeBlock2017,PfirschSobolevPeriodic2017} that fit in the larger class of ergodic operators. For a, say, $\Z^d$-ergodic and selfadjoint operator $\omega\to H_\omega$ on $L^2(\R^d)$, a natural generalization for the left-hand side of \eqref{intro:WidFullAsymp} is the trace of the operator $h(g(H_\omega)_{[-L,L]^d})$ for a suitable function $g$. In \cite{PasturKirschErgodic} such trace asymptotics were studied for one-dimensional random and quasiperiodic Schr\"odinger operators on the lattice. For the random Anderson model and concrete choices of functions $g,h$ the authors showed that the leading order term, which is of order $L$, obeys a central limit theorem. Hence an additional Gaussian fluctuation of order $\sqrt{L}$ can contribute to the asymptotic expansion. Moreover, they exemplified that spectral localization can suppress the logarithmic enhancement \eqref{intro:logenh}. The latter point was generalized in \cite{PasturSlavinAreaLawScaling,PasElgSchLargeBlock2017} to the random Anderson model on the lattice in arbitrary dimension and a larger classes of functions $g,h$. On the other hand, in \cite{PfirschSobolevPeriodic2017} it was proved that the logarithmic enhancemet \eqref{intro:logenh} does occur for one-dimensional periodic continuum Schr\"odinger operators. Those findings are in line with the heuristics that for a logarithmically enhanced subleading term to pop up, a function $g$ with a discontinuity within a conducting energy region of the Hamiltonian $H$ is needed.

In this paper we establish full trace asymptotics as in \eqref{intro:WidFullAsymp} for a large class of selfadjoint $\mathbb{Z}^d$-ergodic operators on $L^2(\R^d)$. Besides some mild general requirements we only impose sufficiently fast decay of the operator kernel of $g(H)$, which can be checked directly in many situations. Typically it stems from either spectral properties of the operator $H$, such as spectral localization, or smoothness properties of the function $g$. We confine ourselves to boxes as scaling domains, $\Omega_L=\Lambda_L := [-L,L]^d$. On the one hand this is necessary because our model is only $\Z^d$-ergodic. But, on the other hand, for only piecewise smooth domains, such as domains with corners, no prior results seem to be available for asymptotic trace expansions beyond the subleading order. In this regard our result is, for example, also new for Wiener-Hopf operators in $d>1$. The restriction to the continuum case is for convenience and analogous results hold for ergodic operators on the lattice $\Z^d$. Let us state an informal version of our result. Precise definitions and statements can be found in Section \ref{sec:ModelResult}. Let $\omega\to H_\omega$ be a $\Z^d$-ergodic selfadjoint operator and let $g:\R\to\R$ be bounded and such that the operator kernel of $g(H)$ decays sufficiently fast. Then, for sufficiently smooth functions $h:\R\to\C$ such that $h(0)=0$,
\begin{equation}
\label{intro:MainResult}
\mathbb{E}\big[ \Tr\big( h(g(H_\omega)_{\Lambda_L}) \big) \big] = \sum_{m=0}^{d} A_{m} (2L)^{d-m} + \mathcal{O}(L^{-\tau}) \qquad (L\to\infty).
\end{equation}
Here $\mathbb{E}$ denotes averaging with respect to $\omega$. Moreover, $\tau>0$ depends on the rate of decay of the operator kernel of $g(H)$ and the regularity of $h$ and in contrast to the expansion \eqref{intro:WidFullAsymp} for smooth domains, the expansion terminates at constant order. The coefficients $A_m$ can be represented as $\omega$-averaged traces of differences of operators of the form $h(g(H_\omega)_G)$, $G\subseteq\R^d$. For more explicit formulas of the coefficients one would have to specialize to concrete models. This can already be seen at the leading-order coefficient $A_0$ which can be interpreted as a density of states term.\\ 
The main idea which leads to \eqref{intro:MainResult} is a scheme of iterated regularizations that allows us to elaborate the contribution of a face of the cube to the different asymptotic orders. Apart from our concrete application this procedure could also be useful in proving asymptotic expansions for more general domains with corners for $\R^d$-ergodic operators. Moreover, parts of the proof could be useful for a more model-dependent analysis. For instance, formula \eqref{pf:SzegoP1P2,15} might serve as a starting point towards a higher-dimensional generalization of the results from \cite{PasturKirschErgodic} on Gaussian fluctuations described above.

The structure of the paper is as follows. In the next section we define our working model, a class of $\Z^d$-ergodic operator on $\R^d$, and present our main result, Theorem \ref{th1:SzegoAsymptAvg}. The theorem is then split into two parts, Theorem \ref{lem:AuxRes1} and Theorem \ref{th:SzegoAsymptAvg}, which are proven in Section \ref{sec:AuxProofs} and Section \ref{sec:ProofTheorem}. At the beginning of those two sections a short outline of the respective proofs is included.

\section{Model and Result}
\label{sec:ModelResult}

\subsection{Model}
\label{sec:TheModel}
Our working model is a $\Z^d$-ergodic operator acting on $L^2(\R^d)$. For a probability space $(\Omega,\mathbb{P})$ this is a measurable map
\begin{equation}
\label{def:TheModel}
\Omega\ni\omega \mapsto H:= H_{\omega} \in \mathcal{L}_{\text{sa}}(L^2(\R^d))
\end{equation}
into the selfadjoint operators on $L^2(\R^d)$ which is $\Z^d$-translation invariant in the following sense. There exists a family $\{T_j\}_{j\in\Z^d}$ of measure preserving (m.p.) transformations $T_j:\Omega\to\Omega$ such that for the unitary family $\{U_j\}_{j\in\Z^d}$ of translation operators on $L^2(\R^d)$, acting as $(U_j\psi)(x)=\psi(x-j)$, we have
\begin{equation}
\label{def:TheModel2}
U_j H_{\omega} U_j^* = H_{T_j \omega}=: H^{T_j}_{\omega}.
\end{equation}  
For details, such as the notion of measurable operators and an interpretation of \eqref{def:TheModel2} in case of domain issues, we refer to \cite{carlac1990random,pastfig1992random}.
As indicated in \eqref{def:TheModel}, the $\omega$-dependence of the operator and related quantities is mostly suppressed in notation. In this vein, \eqref{def:TheModel2} reads $U_jHU_j^* = H^{T_j}$. We impose the following further requirement on the model. For $a\in\R^d$ the operator $\chi_a$ denotes the spatial projection onto the set $Q_a := a+[-2^{-1},2^{-1}]^d$. For $p>0$ we denote the Schatten-$p$ (quasi-)norm of an operator $A$ by $\|A\|_p$.
\be{MyDescription}
\item[$(\mathcal{A}_{1})$]{For any $p>0$ and every (measurable) bounded and compactly supported function $g:\R\to\R$
\begin{equation}
\label{eq:AssagStat}
 C_{p,g} := \sup_{a,b\in\R^d} \esssup_{\omega\in\Omega} \|\chi_a g(H_{\omega})\chi_b \|_p  < \infty.
\end{equation}
\label{AssAg}}
\e{MyDescription}
\begin{remarks}
\item The bound \eqref{eq:AssagStat} for instance holds for Schr\"odinger operators $H=-\Delta+V$ with, for simplicity, bounded potential $V$ \cite[App. A]{artRSO2006AizEtAl2}. A more detailed proof is contained in \cite{DiGeMu16b}.
\item For fixed pairs of functions $g,h$ in Theorems \ref{lem:AuxRes1} and \ref{th:SzegoAsymptAvgstat3} it is sufficient to assume \eqref{eq:AssagStat} for the function $g$ and a sufficiently small $h$-dependent value of $0<p<1$.
\end{remarks}
The $\Z^d$-ergodicity and \ref{AssAg} are the core assumptions. For such operators we prove our main result under the assumption that the operator kernel of $g(H)$ has sufficient spatial decay. A precise notion of this is given in the next section. Besides those essential requirements we facilitate life by introducing additional symmetry. 

\be{MyDescription}
\item[$(\mathcal{A}_2)$]{Symmetry of spatial directions: For $\pi\in \mathcal{S}^d$, the group of permutations on $\{1,...d\}$, we define the unitary operator $U_{\pi}$ on $L^2(\R^d)$ acting as $(U_{\pi}\psi)(x):= \psi(x_{\pi})$, where $x_{\pi}:= (x_{\pi(1)},...,x_{\pi(d)})$. Then for any $\pi\in \mathcal{S}^d$ there exists a m.p. transformation $P_{\pi}$ such that
\begin{equation}
U_{\pi}HU_{\pi}^* = H^{P_{\pi}}.
\end{equation}
 \label{assV2}}
\item[$(\mathcal{A}_3)$]{Reflection symmetry: 
For $\sigma=(\sigma_i)_{i=1}^d \in \{0,1\}^d=:\mathcal{R}^d$ we define the unitary operator $U_{\sigma}$ on $L^2(\R^d)$ acting as $(U_{\sigma}\psi)(x):= \psi(x_{\sigma})$, where $x_{\sigma}:=((-1)^{\sigma_1}x_1,...,(-1)^{\sigma_d}x_d)$. Then for any $\sigma\in \mathcal{R}^d$ there exists a m.p. transformation $R_{\sigma}$ such that 
\begin{equation}
\label{eq:assV2}
U_{\sigma}HU_{\sigma} = H^{R_{\sigma}}.
\end{equation} 
\label{assV3}}
\e{MyDescription}
In \eqref{eq:assV2} we used that $U_\sigma= U_\sigma^*$ for $\sigma\in \mathcal{R}^d$.
Those two additional assumptions are made for convenience and could be dropped. We included them because they make statement and proof of our results less cumbersome; for instance \ref{assV2} allows to reduce up to $|S_d|=d!$ terms to only one. Our guiding example of metrically transitive operators are Schr\"odinger operators $H=-\Delta+V$, where $-\Delta=-\sum_{j=1}^d \partial^2_{x_j}$ is the Laplacian and $V=V_\omega(x)$ is an $\omega$-dependent and real-valued potential that satisfies $U_jV_\omega(\cdot)U_j^* = V_{T_j\omega}(\cdot)$. Concrete examples are periodic Schr\"odinger operators and random alloy-type Schr\"odinger operators with periodically arranged single-site potentials.

Let us fix some notation before starting off. The integral with respect to $\mathbb P$ is denoted by $\E$ and referred to as expectation. On $\R^d$ we always consider the supremum norm $|x|:=\sup_{j=1}^d |x_j|$. The distance function on $\R^d$ with respect to supremum norm, either of a set and a point or of two sets, is denoted by $\distt(\cdot, \cdot)$. The indicator function of a set $G$ is denoted by $\id_G$ and the corresponding orthogonal projection acting on $L^2(\R^d)$ is denoted by $\chi_G$.  For a bounded function $g$ (here, and in the following, all functions are assumed to be measurable) we write $g(H)_G$ and $g(H)_L$ for the restriction of $g(H)$ onto $G$ and $\Lambda_L := [-L,L]^d$, respectively. If convenient, we view the bounded operators $g(H)_G$ on $G\subseteq\R^d$ as operators on $\R^d$ via the natural embedding. With some abuse of notation, this is $g(H)_G = \chi_G g(H) \chi_G$. Moreover, $Q := [-2^{-1},2^{-1}]^d$ and $Q_a:= a+Q$ are the cubes of side-length $1$ around $0$ respectively $a\in\R^d$ and $\chi_a:= \chi_{Q_a}$ denotes the corresponding projection operator. For a set such as $\{x\in\R^d: x_1,...,x_k\in [0,1]\}$, where $1\leq k \leq d$, we occasionally write $\id_{\{x_1,...,x_k\in [0,1]\}}$ and $\chi_{\{x_1,...,x_k\in [0,1]\}}$ for the corresponding indicator function and projection operator.

\subsection{Main result}
\label{sec:MainResults}

Besides the technical properties \ref{AssAg}--\ref{assV3} introduced above, our main assumption is sufficiently fast decay of the operator kernel of $g(H)$, where $g:\R\to\R$ is a compactly supported and bounded function. The two guiding examples for which decay of the operator kernel is known are spectral localization of the operator $H$ and sufficiently regular functions $g$. To cover those two guiding examples we assume that one of the following two conditions holds.
\be{MyDescription}
\item[$(\mathcal{L}_{1,q})$]{For fixed $q>0$ there exists a constant $C_{1,q}$ such that for all $a,b\in \R^d$ 
\begin{equation}
\label{eq:AssLoc1}
\esssup_{\omega\in\Omega} \|\chi_a g(H_{\omega}) \chi_b\| \leq \frac{C_{1,q}}{\left(1+|a-b|\right)^q}.
\end{equation} \label{AssLoc1}}
\item[$(\mathcal{L}_2)$]{There exist constants $C_2,\mu>0$ such that for all $a,b\in \R^d$
\begin{equation}
\label{eq:AssLoc2}
 \mathbb{E} \left[\|\chi_a g(H) \chi_b\| \right] \leq C_2 e^{-\mu |a-b|}.
\end{equation} 
 \label{AssLoc2}}
\e{MyDescription}
The first condition holds for a large class of operators which obey a Combes-Thomas estimate and with a power of $q$ that depends on the regularity of $g$. More concretely, if $H=-\Delta+V$ is a Schr\"odinger operator with, for simplicity, uniformly (in $x\in\R^d$ and in $\omega\in\Omega$) bounded potential $V$, then $g\in \mathcal{C}_c^{q+2}(\R)$ implies that \ref{AssLoc1} holds \cite{MR1937430}. The second bound for instance holds in case $H$ is an alloy-type random Schr\"odinger operator and $g$ is a bounded function such that $\supp(g)$ is a subset of the region of spectral localization characterized via fractional moment bounds \cite{artRSO2006AizEtAl2}. 

To state the main result, and to define the asymptotic coefficients from \eqref{intro:MainResult}, we introduce some more notation. For $n=0,...,d$ we define the model operators
\begin{equation}
\label{th:SzegoAsymptAvgstat3}
f_{n} := h(g(H)_{\mathbb{R}_{\geq 0}^n\times \R^{d-n}}),
\end{equation}
which approximate $h(g(H)_{\Lambda_L})$ in respective areas of the cube $\Lambda_L$. For instance, $f_0$ is an approximation of the operator in the bulk of $\Lambda_L$ and $f_1$ is an approximation of the operator along a face of $\Lambda_L$ (taking the symmetries \ref{assV2} and \ref{assV3} into account). 
Moreover, for $1\leq n \leq m\leq d$ we set
\begin{align}
c_{m,n} &:= \frac{(-1)^{m-n}2^{m}d!}{(m-n)!(d-m)!},\\
\widehat{\chi}_{m,n} &:= \chi_{\R_{\geq 0}^d}\chi_{\{x_1\leq...\leq x_{n}\}} \chi_{\{x_n\geq x_{n+1},...,x_{m}\}}\chi_{\{x_{m+1},...,x_{d}\in [0,1]\}} \label{def:const1}.
\end{align}
The constant $c_{m,n}$ is a combinatorial factor which stems from collecting terms via the symmetry assumptions \ref{assV2} and \ref{assV3}. The projection operator $\widehat{\chi}_{m,n}$ ensures that the first $n$ coordinates are ordered increasingly and, in addition, that the $n$-th coordinate is larger than the first $m$ coordinates. If $n=m$, we interpret $\chi_{\{x_n\geq x_{n+1},...,x_{m}\}} = \idd_{L^2(\R^d)}$ in \eqref{def:const1}, and, in the same vein, $\chi_{\{x_{d+1},...,x_{d}\in [0,1]\}} = \idd_{L^2(\R^d)}$ if $m=d$. Finally, for a fixed bounded function $g:\R\to\R$ we set
\begin{equation}
\label{def:whsigma}
\widehat{\Sigma}_g := [\inf \Sigma_{g(H)},\sup \Sigma_{g(H)}],
\end{equation}
where $\Sigma_{g(H)}$ is the (almost surely non-random) spectrum of $g(H)$. 

\begin{theorem}
\label{th1:SzegoAsymptAvg}
Let $g:\R\to\R$ and $h:\R\to\C$ be two compactly supported and bounded functions with $h(0)=0$. If one of the following two conditions is satisfied for $\widetilde{q}>2d$
\begin{enumerate}
\item[(i)]  \ref{AssLoc2} holds and $h\in \mathcal{C}^{\lfloor 2\widetilde{q}+2\rfloor}(\R)$,
\item[(ii)] \ref{AssLoc1} holds for $q> 2d+\widetilde{q}$ and $h$ can be continued analytically to $\{z\in \C:\, \dist(z,\widehat{\Sigma}_g) < C_{g,\widetilde{q}}\}$ for the constant $C_{g,\widetilde{q}}$ specified in \eqref{eq:HoloConst} below,
\end{enumerate}
then, as $\N\ni L\to\infty$, the asymptotic expansion
\begin{equation}
\label{cor:SzegoAsymptAvgStat}
\mathbb{E}\big[ \Tr\big( h(g(H)_{\Lambda_L}) \big) \big] = \sum_{m=0}^{d} A_{m} (2L)^{d-m} + \mathcal{O}(L^{2d-\widetilde{q}})
\end{equation} 
holds. The coefficients $A_m$ are defined as
\begin{align}
\label{cor:SzegoAsymptAvgstat2}
A_{m} &:= \sum_{n=1}^m c_{m,n}  \Tr\left(\mathbb{E}\left[\widehat{\chi}_{m,n} \{ f_{n}-f_{n-1} \} \widehat{\chi}_{m,n} \right] \right).
\end{align}
\end{theorem}

\begin{remark}
\label{rem:SzAsym1}
The representation \eqref{cor:SzegoAsymptAvgstat2} of the coefficients is not unique and depends on the partition of corners for the cube $[-L,L]^d$ which we choose in the proof. At the end of Section \ref{sec:ProofTheorem} we show that the coefficients also have a partition-free representation:
\begin{align}
\label{cor:rmk1}
A_{m} &= \lim_{L\rightarrow\infty} \sum_{n=0}^{m} \widetilde{c}_{m,n} \mathbb{E}\left[ \Tr\left( f_{n} \chi_{[0,L]^d}\chi_{\{x_{m+1},..,x_{d}\in[0,1]\}} \right) \right]
\end{align}
for constants $\widetilde{c}_{m,n}$ defined in \eqref{eq:Part3E4}. The operator $f_n \chi_{[0,\infty)^d} \chi_{\{x_{m+1},...,x_{d}\in [0,1]\}}$ is trace class only if $m=0$, which corresponds to the coefficient $A_0$. The $L$-limit can therefore not be interchanged with the sum appearing in \eqref{cor:rmk1}.
\end{remark}

\begin{remarks}
\item The validity of the asymptotic expansion \eqref{cor:SzegoAsymptAvgStat} is not restricted to assumptions $(i)$ or $(ii)$, which rather serve as two relevant examples. See Remark \ref{rem:ValidityTraceEst} below.
\item The uncommon ordering of expectation and trace norm in \eqref{cor:SzegoAsymptAvgstat2} stems from Lemma \ref{lem:CTRan2} and is only necessary under assumption $(i)$.
\item \label{rem:SzAsym3} Under reasonable assumptions, the theorem can be extended to non-integer length-scales $L\in \R_{>0}$. For $\Z^d$-ergodic operators $H$ as considered here the coefficients $A_m$ then become functions of the fractional part of $L$. This dependence in turn does not show up if the operator is invariant under $\R^d$-translations. 
\end{remarks}
The proof of Theorem \ref{th1:SzegoAsymptAvg} is presented in two parts, Theorem \ref{lem:AuxRes1} and Theorem \ref{th:SzegoAsymptAvg} below. The aim of this subdivision is to split the result into an analytic part, Theorem \ref{lem:AuxRes1}, and an algebraic part, Theorem \ref{th:SzegoAsymptAvg}. 
For sets $G\subset G'\subseteq \R^d$ we define the boundary of $G$ in $G'$ as $\partial_{G'}G:= \partial G \cap \intt(G')$, where $\intt(\cdot)$ denotes the topological interior.

\begin{theorem}
\label{lem:AuxRes1}
Let $g:\R\to\R$ and $h:\R\to\C$ be two compactly supported, bounded functions with $h(0)=0$. If, additionally, one of the following two conditions is satisfied for fixed $\widetilde{q}>0$ 
\begin{enumerate}
\item[(i)]  \ref{AssLoc2} holds and $h\in \mathcal{C}^{\lfloor 2\widetilde{q}+2\rfloor}(\R)$,
\item[(ii)] \ref{AssLoc1} holds for $q> 2d+\widetilde{q}$ and $h$ can be continued analytically to $\{z\in \C:\, \dist(z,\widehat{\Sigma}_g) < C_{g,\widetilde{q}}\}$ for the constant $C_{g,\widetilde{q}}$ specified in \eqref{eq:HoloConst} below,
\end{enumerate}
then the following holds: 
There exists a constant $C_{g,h,\widetilde{q}}$ such that for all $G\subset G'\subseteq\R^d$ and all $a,b\in G'$ with $Q_a\subset G$ or $Q_b\subset G$ 
\begin{equation}
	\label{lem:AuxRes1Stat}
\left\Vert \mathbb{E} \left[\chi_a \{ h(g(H)_G) - h(g(H)_{G'})\} \chi_b  \right]\right\Vert_1  \leq  \frac{C_{g,h,\widetilde{q}}}{\distt(a,\partial_{G'}G)^{\widetilde{q}}+\distt(b,\partial_{G'}G)^{\widetilde{q}}}.
\end{equation}
\end{theorem}

\begin{theorem}
\label{th:SzegoAsymptAvg}
Let $g:\R\to\R$ and $h:\R\to\C$ be two compactly supported and bounded functions such that there exist constants $C_h,\gamma_h>0$ such that $|h|\leq C_h| \cdot |^{\gamma_h}$. If \eqref{lem:AuxRes1Stat} holds for $\widetilde{q}>2d$, then, as $\N\ni L\to\infty$
\begin{equation}
\label{cor2:SzegoAsymptAvgStat}
\mathbb{E}\big[ \Tr\big( h(g(H)_{\Lambda_L}) \big) \big] = \sum_{m=0}^{d} A_{m} (2L)^{d-m} + \mathcal{O}(L^{2d-\widetilde{q}}).
\end{equation} 
\end{theorem}

\begin{remarks}
\item  If \ref{AssLoc2} holds for a deterministic model, $\Omega = \{0\}$, then the proof of Theorem \ref{lem:AuxRes1} and Remark \ref{rem:CTnonopt} show that $h\in \mathcal{C}^{\widetilde{q}+1}(\R)$ implies \eqref{lem:AuxRes1Stat}. This is probably also true for the general case but would require a refined version of the Combes-Thomas estimate from Lemma \ref{lem:CTRan2}.
\item Under assumption $(ii)$, the expectation in \eqref{lem:AuxRes1Stat} is obsolete.
\item For the special case of random alloy-type Schr\"odinger operators and a function $g$ such that $\supp(g)$ is a subset of the region of spectral localization characterized via fractional moment bounds, \eqref{lem:AuxRes1Stat} seems to be a weak conclusion from \ref{AssLoc2}: It is for instance known that in this case $\mathbb{E} \left[\|\chi_a (h\circ g)(H) \chi_b\| \right]$ is exponentially decaying in $|a-b|$ for any bounded function $h$. But in order to conclude \eqref{lem:AuxRes1Stat} without any smoothness assumption on $h$ one would have to rule out extended boundary states for the random operator $g(H)$. To the author's knowledge this is not known in such generality for $d>1$. \label{rem:TraceEstAndLoc}
\item The bound \eqref{lem:AuxRes1Stat} is not restricted to the assumptions $(i)$ and $(ii)$ in Theorem \ref{lem:AuxRes1}. A special yet very different scenario is the following. If $H$ is a random alloy-type Schr\"odinger operator and we take $g=\idd_\R$ and $h=\id_{(-\infty,E]}$ for an energy $E$ within the region of spectral localization characterized via fractional moment bounds, then \eqref{lem:AuxRes1Stat} holds with exponential decay in $\distt(a,\partial_{G'}G)$ and $\distt(b,\partial_{G'}G)$ \cite{DiGeMu16b}. From the perspective of Remark \ref{rem:TraceEstAndLoc} above, this is the trivial case in which extended boundary states for $g(H)$ can be ruled out in the relevant spectral region. \label{rem:ValidityTraceEst}
\end{remarks}

\section{Proof of Theorem \ref{lem:AuxRes1}}
\label{sec:AuxProofs}

For each of the assumptions $(i)$ and $(ii)$ we first prove an operator-norm version of the estimate \eqref{lem:AuxRes1Stat} and in both cases we employ a suitable functional calculus to rewrite $h(g(H)_{G^{(\prime)}})$ in terms of the resolvent of $g(H)_{G^{(\prime)}}$. Via a Combes-Thomas estimate and the geometric resolvent equation we then localize the operator $h(g(H)_G) - h(g(H)_{G'})$ to $\partial_{G'}G$. Finally, the corresponding trace-norm estimate follows from interpolation with Schatten-$p$ bounds ($p<1$) for the difference of operators on the left-hand side of \eqref{lem:AuxRes1Stat}. Such bounds are a consequence of \ref{AssAg}, see Lemma \ref{lem:AppTrClEst} below. This argument is the same for both assumptions and is only carried out for case $(i)$. 

For assumption \ref{AssLoc1} a polynomial Combes-Thomas estimate for the resolvent of $g(H)_G$, $G\subseteq \R^d$, is known to hold \cite{Aizenman93localizationat}. The holomorphic functional calculus then lifts this mild decay of the resolvent to decay of the operator $h(g(H)_G) - h(g(H)_{G'})$.

In case of assumption \ref{AssLoc2} only decay of the averaged operator kernel is known. While this seems to shut down the standard approach for the Combes-Thomas estimate we show below that an alternative approach - power series expansion of the resolvent far apart from the spectrum and subsequent interpolation in the complex energy parameter - is flexible enough. To the best of our knowledge this approach is not covered in the literature. This is why we included a detailed proof in the next section. Once the Combes-Thomas estimate is established we apply the Helffer-Sj\"ostrand formula to rewrite $h(g(H)_G)$ in terms of the resolvent of $g(H)_G$. This final step is essentially contained in \cite{MR1937430}. For convenience we included some of its details in Section \ref{sec:HelfSjo} below.

Let $g:\R\to\R$ be bounded and compactly supported. For the whole section we abbreviate $A:= g(H)$ and $\widehat\Sigma=\widehat\Sigma_g$. The restriction of the bounded operator $A$ to $G\subset\R^d$ is denoted by $A_G$ and we write $R_z (A_G)=(A_G-z)^{-1}$ for the operators resolvent at $z\in\C\setminus \sigma(A_G)$. In the following we stick to our original setup but one can think of $A$ as an arbitrary bounded ($\omega$-dependent) operator satisfying \eqref{eq:AssagStat} and either \eqref{eq:AssLoc1} or \eqref{eq:AssLoc2}.

\subsection{Proof of Theorem \ref{lem:AuxRes1} under assumption $(ii)$}
\label{sec:AuxProofs1st}
For operators on $\Z^d$ a polynomial Combes-Thomas estimate is proved in \cite[App. II]{Aizenman93localizationat} and reviewed in \cite{AizWarBook}. Their proof carries over to our setup. For the next few lines the notation closely sticks to \cite{AizWarBook}. If matrix elements are substituted by operator kernels $\chi_a (A_G-z)^{-1} \chi_b$ for $a,b\in\Z^d\cap G$, then the proof works if we choose a distance function which is constant on unit cubes $Q_a$, $a\in\Z^d$. The transition to arbitrary $a,b\in G$ then induces a slightly enlarged constant in \eqref{eq:CTConcl} below. The term $(|a-b|+2)^{q'}$ in \eqref{eq:HoloConst} below instead of $(|a-b|+1)^{q'}$ in \cite{AizWarBook} is due to the transition from the $\Z^d$-adapted distance to the original distance. Let $\varepsilon>0$ such that $q=\widetilde{q}+2d+\varepsilon$ and define $q' = q-d-\varepsilon/2 = \widetilde{q} + d + \varepsilon/2$. Then, via the polynomial Combes-Thomas estimate, 
\begin{equation}
\label{eq:CTConcl}
\| \chi_a R_z(A_G)\chi_b \| \leq  \frac{C_1}{(|a-b|+1)^{q'}}
\end{equation}
holds for all $z\in \C$ that satisfy
\begin{equation}
\label{eq:HoloConst}
\dist(z,\widehat{\Sigma}) \geq 1+ \sup_{a\in\Z^d} \sum_{b\in\Z^d} \|\chi_a A_G\chi_b\| \left( (|a-b|+2)^{q'}-1 \right)=: C_{g,\widetilde{q}}-1.
\end{equation}
Fix $a,b\in G'$ such that $Q_a \subset G$. By assumption the function $h$ can be continued analytically onto $\{z\in\C:\, \dist(z,\widehat{\Sigma})  < C_{g,\widetilde{q}}\}$. Let $\Gamma$ be a smooth oriented curve, with winding number $=1$ for the set $\widehat{\Sigma}$, such that 
\begin{equation}
\Ran(\Gamma) \subset \{z\in\C:\, C_{g,\widetilde{q}}-1 < \dist(z,\widehat{\Sigma}) < C_{g,\widetilde{q}}\}
\end{equation}
holds for the range of $\Gamma$. The holomorphic functional calculus then yields
\begin{equation}
\label{eq:Holo1}
\chi_a \left(h(A_G)-h(A_{G'})\right) \chi_b = \frac{1}{2\pi i} \int_{\Gamma}\mathrm{d}z\, h(z) \chi_a R_z(A_{G}) \left(  \chi_G A \chi_{G'\setminus G}\right) R_z(A_{G'})\chi_b,
\end{equation}
where we also applied the geometric resolvent equation and used $Q_a\subset G$. For a set $U\subseteq\R^d$ we define $U_+:= \{n\in (\Z+1/2)^d:\, Q_n\cap U \neq \emptyset\}$. The operator norm of \eqref{eq:Holo1} can then be estimated as
\begin{align}
\label{eq:Holo2}
\left\Vert \eqref{eq:Holo1} \right\Vert &\leq C_2 \sum_{\substack{l\in G_+ \\k\in (G'\setminus G)_+}} \frac{1}{(|a-l|+1)^{q'}}\frac{1}{(|l-k|+1)^{q}} \frac{1}{(|k-b|+1)^{q'}} \notag\\
&\leq \frac{C_3}{\distt(a,\partial_{G'}G)^{q'}},
\end{align}
where we used that $q'>d$ and $q>q'+d$. Because the same bound holds with $b$ instead of $a$ on the right-hand side of \eqref{eq:Holo2} we obtain
\begin{equation}
\label{eq:Holo3}
\left\Vert \eqref{eq:Holo1} \right\Vert \leq \frac{C_4}{\distt(a,\partial_{G'}G)^{q'}+\distt(b,\partial_{G'}G)^{q'}}.
\end{equation}
Finally we interpolate \eqref{eq:Holo3} with Schatten-class bounds for the operator kernel of $h(g(H)_G)$. Such bounds follow from \ref{AssAg} and the next lemma.

\begin{lemma}
	\label{lem:AppTrClEst}
	Let $B$ be a selfadjoint bounded operator on $L^2(\R^d)$ such that
	\begin{equation}
	\label{lem:AppTrClEstAspt}
		C_{B,p} := \sup_{ a\in\R^d } \| \chi_a B^2 \chi_a \|_p <\infty
	\end{equation}
	holds. Then, for functions $h:\R\to\C$ such that $|h|\leq C_h| \cdot |^{\gamma_h}$ holds for constants $C_h$ and $0<\gamma_h\leq 1$,
	\begin{equation}
		\label{lem:AppTrClEstStat}
		\sup_{G\subseteq\R^d} \sup_{ a\in\R^d } \| \chi_a h(B_G) \|_{\frac{2p}{\gamma_h}} \leq C_{B,p}^{p} C_h^{\frac{2p}{\gamma_h}}.
	\end{equation}
\end{lemma}
The proof of the Lemma is given below. For $p>\delta>0$ and an operator $A$ the bound
	\begin{align}
		\|A\|_{p}^p \leq \|A\|^{\delta} \|A\|_{p-\delta}^{p-\delta}
	\end{align}
	holds. With $p=1$ and $\delta=\widetilde{q}/q'\in (0,1)$ we obtain
	\begin{align}
	\left\Vert  \chi_a \{ h(A_G) - h(A_{G'})\} \chi_b  \right\Vert_1 &\leq \left\Vert \chi_a \{ h(A_G) - h(A_{G'})\} \chi_b  \right\Vert^{\widetilde{q}/q'}\notag\\
	&\quad \times \left\Vert \chi_a \{ h(A_G) - h(A_{G'})\} \chi_b  \right\Vert_{1-\widetilde{q}/q'}^{1-\widetilde{q}/q'}. 	\label{eq:CorGeomS2}
	\end{align}
	The first term on the right-hand side of \eqref{eq:CorGeomS2} can be estimated via \eqref{eq:Holo3}. For the second term we apply Lemma \ref{lem:AppTrClEst} to the operator $A^2$. Assumption \ref{AssAg} ensures that \eqref{lem:AppTrClEstAspt} holds and the bound on $h$ follows from smoothness and $h(0)=0$. The lemma yields
	\begin{align}
	\left\Vert \chi_a \{ h(A_G) - h(A_{G'})\} \chi_b  \right\Vert_{1-\widetilde{q}/q'}^{1-\widetilde{q}/q'} &\leq \left\Vert\chi_a  h(A_G) \chi_b \right\Vert_{1-\widetilde{q}/q'}^{1-\widetilde{q}/q'} +\left\Vert\chi_a  h(A_{G'})\chi_b \right\Vert_{1-\widetilde{q}/q'}^{1-\widetilde{q}/q'}\notag \\
	&\leq C_5
	\end{align}
and overall we found that
\begin{align}
\left\Vert \chi_a \{ h(A_G) - h(A_{G'})\} \chi_b  \right\Vert_1 \leq \frac{C_6}{\distt(a,\partial_{G'}G)^{\widetilde{q}}+(
\distt(b,\partial_{G'}G)^{\widetilde{q}}}.
\end{align}
for $a,b\in G'$ such that $Q_a\subset G$. The proof for $Q_b\subset G$ follows along the same lines.
\qed

\begin{proof}[Proof of Lemma \ref{lem:AppTrClEst}]
	The singular values of $\chi_a h(B_G)$ are
		\begin{align}
			\mu_n(\chi_a h(B_G)) &= \sqrt{\lambda_n(\chi_a |h|^2(B_G) \chi_a )}.		\label{eq:AppTrClEst1}
		\end{align}
Moreover, because $|h| = | \cdot |^{\gamma_h} \widetilde{h}$ for some non-negative function $\widetilde{h}$ which is bounded by $C_{h}$, the form inequality
		\begin{align}
			\chi_a  |h|^2(B_G) \chi_a &= \chi_a \vert B_G \vert^{\gamma_h} \widetilde{h}^2(B_G)\vert B_G \vert^{\gamma_h} \chi_a\notag\\
			&\leq C_h^2 \chi_a \vert B_G \vert^{2\gamma_h} \chi_a \label{eq:AppTrClEst2}
		\end{align}
	holds. The function $x\to x^{\gamma_h}$ is operator monotone because $0<\gamma_h<1$. Hence the form inequality
	\begin{equation}
		\label{eq:AppTrClEst3}
		|B_G|^{2\gamma_h} = (\chi_G B \chi_G B \chi_G)^{\gamma_h} 
		\leq  ( \chi_G B^2 \chi_G )^{\gamma_h} 
	\end{equation}
	holds. \eqref{eq:AppTrClEst2} and \eqref{eq:AppTrClEst3} together with the bound
	\begin{equation}
		\left\langle \psi, \chi_a(\chi_G B^2  \chi_G)^{\gamma_h}\chi_a \psi  \right\rangle  \leq \left\langle \psi, \chi_a\chi_G B^2 \chi_G\chi_a \psi  \right\rangle ^{\gamma_h}
	\end{equation}
	for normalized $\psi\in L^2(\R^d)$ yield
	\begin{equation}
		\label{eq:AppTrClEst4}
		\lambda_n(\chi_a |h|^2(B_G) \chi_a ) \leq C_h^2 \lambda_n(\chi_a\chi_G B^2\chi_G\chi_a)^{\gamma_h}.
	\end{equation}
	Let $p' = 2p/\gamma_h$. Together with \eqref{eq:AppTrClEst1} this yields
		\begin{align}
			\|\chi_a h(B_G)\|_{p'}^{p'} &= \sum_{n\in\N} \mu_n(\chi_a h(B_G))^{p'} \leq C^{p'}_h \sum_{n\in\N}\lambda_n(\chi_a\chi_G B^2\chi_G\chi_a)^{p'\gamma_h/2}\notag \\
			&\leq C^{2p/\gamma_h}_h \|\chi_a B^2\chi_a \| _{p}^{p}.
		\end{align} 
\end{proof}

\subsection{Combes-Thomas estimate under assumption \ref{AssLoc2}}

In this section we prove that averaged decay of the operator kernel of $A$ is sufficient to deduce averaged decay for the operator kernel of the resolvent at complex energies away from the spectrum. We state two different versions of this result, Lemma \ref{lem:CTRan} and Lemma \ref{lem:CTRan2}. The first Lemma is not needed for the proof of Theorem \ref{lem:AuxRes1} $(i)$ but serves to illustrate the method and can be directly compared to the classical Combes-Thomas estimate. Detailed proofs are included because, to the best of our knowledge, this approach is not covered in the literature.

\begin{lemma}
	\label{lem:CTRan} Assume that \ref{AssLoc2} holds and let $0<\theta<1/2$ be fixed. Then there exist constants $C_{\theta},\mu_{\theta}>0$ such that for $z\in \C\setminus \widehat{\Sigma}$
	\begin{equation}
	\label{cor:CTRanStat}
	\left\Vert \mathbb{E} \left[  \chi_a R_z(A_G) \chi_b  \right]\right\Vert\leq \frac{C_{\theta}}{\distt (z,\widehat{\Sigma})} e^{-\mu_{\theta}\distt(z,\widehat{\Sigma}) |a-b|^{\theta}}.
	\end{equation}
\end{lemma}
For sets $G\subset G'\subseteq\R^d$ we recall the definition $\partial_{G'}G:= \partial G \cap \intt(G')$ for the boundary of $G$ in $G'$.

\begin{lemma}
\label{lem:CTRan2}
Assume that \ref{AssLoc2} holds and let $0<\theta<1/2$ be fixed. Then there exist constants $C_{\theta},\mu_{\theta}>0$ such that for $G\subset G'\subseteq\R^d$ and $a,b\in G'$ with $Q_a\subset G$ or $Q_b\subset G$ the bound 
\begin{align}
\left\Vert \mathbb{E} \left[  \chi_a \left(R_z(A_G)-R_z(A_{G'})\right) \chi_b  \right]\right\Vert \leq  \frac{C_{\theta}}{\distt(z,\widehat{\Sigma})} e^{-\mu_{\theta}\distt(z,\widehat{\Sigma}) \left(\distt(a,\partial_{G'}G)^{\theta}+\distt(b,\partial_{G'}G)^{\theta}\right)}
\label{cor:CTRanStat2}
\end{align}
holds for all $z \in \C\setminus \widehat{\Sigma}$. 
\end{lemma}
\begin{remark}
 The reason why only fractional exponential decay is established stems from the rather bold application of H\"older's inequality in \eqref{eq:PfCtRan4} below. 
For a deterministic model, i.e. $\Omega = \{0\}$, the proof yields exponential decay ($\theta=1$ in \eqref{cor:CTRanStat} and \eqref{cor:CTRanStat2}).  \label{rem:CTnonopt} 
\end{remark}

\begin{proof}[Proof of Lemma \ref{lem:CTRan}]
For convenience we fix $\theta=1/4$ for the proof.
Let $a,b\in \R^d$ be fixed. Then, for fixed $m$ with $0<m<M:=2(\|A\|_{\infty}+1)$,
\begin{equation}
\{z\in\C:\ m<\Im(z)<M\}=:S_{m,M}\ni z \mapsto f(z):= \mathbb{E}\left[\chi_a R_z(A)\chi_b \right],
\end{equation} 
is an operator-valued analytic map which is continuous on $\overline{S_{m,M}}$ and bounded by $1/m$. For $m\leq t \leq M$ we define
\begin{equation}
F_t := \sup_{x\in\R} \left\Vert f(x+it)\right\Vert \leq \frac{1}{t}.
\end{equation}
Then the Stein interpolation theorem \cite{bennett1988interpolation} states that for $m\leq t \leq M$ the bound
\begin{equation}
\label{eq:SteinInterpolStat}
F_t \leq F_m^{\frac{M-t}{M-m}} F_M^{\frac{t-m}{M-m}}
\end{equation}
holds, where $F_m$ can be estimated by $1/m$. 	
In order to estimate $F_M$ we expand the resolvent $R_z(A)$ as a Neumann series. This yields
\begin{align}
F_M = \sup_{x\in\R} |f(x+iM)| &\leq \sum_{l=0}^{N} \frac{\mathbb{E}\left[\Vert \chi_a A^l \chi_b \Vert\right]}{2^{l+1}(\|A\|_{\infty}+1)^{l+1}} + \sum_{l=N+1}^{\infty} \frac{\mathbb{E}\left[\Vert \chi_a A^l \chi_b \Vert \right]}{2^{l+1}(\|A\|_{\infty}+1)^{l+1}}\notag\\
&=: I_1+I_2
\label{eq:PfCtRan2}
\end{align}
for some $N>0$ which is specified below. We estimate $I_2$ as
\begin{equation}
\label{eq:PfCtRan3}
I_2\leq \sum_{l=N+1}^{\infty} \frac{\|A\|_{\infty}^l}{2^{l+1}(\|A\|_{\infty}+1)^{l+1}} \leq \frac{1}{2^N}.
\end{equation} 
To estimate $I_1$ we set, for fixed $l>0$, $k_0:=a$ and $k_l:=b$. An application of H\"older's inequality then yields
\begin{align}
\mathbb{E}\left[\Vert \chi_a A^l \chi_b \Vert\right] &\leq \sum_{k_1,...,k_{l-1}\in \Zd} \mathbb{E}\Big[\prod_{j=1}^{l} \Vert \chi_{k_{j-1}} A \chi_{k_j} \Vert \Big] \notag\\
&\leq \sum_{k_1,...,k_{l-1}\in \Zd} \prod_{j=1}^{l} \mathbb{E}\left[ \Vert \chi_{k_{j-1}} A \chi_{k_j} \Vert^l\right]^{1/l} \notag\\
&\leq C_2^l\|A\|_{\infty}^{l-1} \sum_{k_1,...,k_{l-1}\in \Zd} \prod_{j=1}^{l} e^{-\frac{\mu}{l}|k_{j-1}-k_j|},
\label{eq:PfCtRan4}
\end{align} 
where, for the last inequality, we used \ref{AssLoc2}. The product in \eqref{eq:PfCtRan4} can be estimated as 
\begin{equation}
\label{eq:PfCtRan5}
\prod_{j=1}^{l} e^{-\frac{\mu}{l}|k_{j-1}-k_j|} \leq e^{-\frac{\mu}{2l}|a-b|} \prod_{j=1}^{l} e^{-\frac{\mu}{2l}|k_{j-1}-k_j|}.
\end{equation}
Let us assume that $|a-b|>1$. Then the sum defining $I_1$ starts at $l=1$ and we obtain the upper bound
\begin{align}
I_1 &\leq  \sum_{l=1}^{N} \frac{\|A\|_{\infty}^{l-1}}{2^{l+1}(\|A\|_{\infty}+1)^{l+1}} e^{-\frac{\mu}{2l}|a-b|} \sum_{k_1,...,k_{l-1}\in \Zd}  \prod_{j=1}^{l} e^{-\frac{\mu}{2l}|k_{j-1}-k_j|}\notag \\
&\leq \sum_{l=1}^{N} 2^{-(l+1)} e^{-\frac{\mu}{2l}|a-b|} \left(\sum_{k\in\Zd} e^{-\frac{\mu}{2l}|k|} \right)^{l-1}.
\label{eq:PfCtRan6}
\end{align}
The $k$-sum on the right-hand side of \eqref{eq:PfCtRan6} can be estimated from above by $Bl^d$
for an $l$-independent constant $B$. 
Hence $F_M$ can be estimated as
\begin{align}
F_M &\leq e^{-N\log(2)}  + e^{-\frac{\mu}{2N}|a-b|} \sum_{l=1}^{N} \frac{(Bl^d)^{l-1}}{2^{l+1}}\notag \\
&\leq  e^{-N \log(2)}  + e^{-\frac{\mu}{2N}|a-b|}  N B^N N^{dN}.
\label{eq:PfCtRan8}
\end{align}
For the choice $N=|a-b|^{1/4}$, this yields
\begin{align}
F_M  &\leq   e^{-\log(2)|a-b|^{1/4}} +  e^{-\frac{\mu}{2}|a-b|^{3/4}} |a-b|^{1/4} e^{|a-b|^{1/4}\left( \log(B)+\frac{d}{4}\log|a-b|\right)} \notag\\
&\leq C e^{-\mu'|a-b|^{1/4}} \label{eq:PfCtRan9}
\end{align}
for $\mu'= \min\{\mu/2, \log(2)\}$.
With \eqref{eq:SteinInterpolStat} for $t=2m$ we arrive at
\begin{equation}
\label{eq:PfCtRan12}
F_{2m} \leq F_m^{\frac{M-2m}{M-m}} F_M^{\frac{m}{M-m}} \leq \left(\frac{1}{m}\right)^{\frac{M-2m}{M-m}} \widetilde{C}^{\frac{m}{M-m}} e^{-\frac{\widetilde{\mu} m}{M-m} |a-b|^{1/4}}.
\end{equation}
For $\eta>0$ this can be written as
\begin{equation}
\label{eq:PfCtRan13}
\sup_{E\in \R} \big\Vert\mathbb{E}\left[ \chi_a R_{E+i\eta}(A) \chi_b \right] \big\Vert \leq \left(\frac{2}{\eta}\right)^{\frac{M-\eta}{M-\eta/2}} C^{\frac{\eta}{2M-\eta}} e^{-\frac{\mu' \eta}{2M-\eta}|a-b|^{1/4}}.
\end{equation}
Because $M\geq 2$ we get for $\eta \in (0,1)$ the more appealing bound
\begin{equation}
\label{eq:PfCtRan14}
\sup_{E\in \R} \big\Vert\mathbb{E}\left[ \chi_a, R_{E+i\eta}(A) \chi_b \right] \big\Vert \leq \frac{\widehat{C}}{\eta} e^{-\widehat{\mu}\eta |a-b|^{1/4}}
\end{equation}
for constants $\widehat{C},\widehat{\mu}>0$ that are independent of $\eta\in (0,1)$ and $a,b\in\R^d$. For $\eta<0$ the same interpolation argument can be performed below the real axis. This yields \eqref{cor:CTRanStat} in case $z=E+i\eta\in \C\setminus \widehat\Sigma$ is such that $\distt(E, \widehat{\Sigma}) \leq |\eta|$. If $\distt(E, \widehat{\Sigma}) \geq |\eta|$, then  \eqref{cor:CTRanStat} would follow from interpolation on a vertical strip. But in this case interpolation is not even needed since the resolvent can directly be expanded.  
\end{proof}

\begin{proof}[Proof of Lemma \ref{lem:CTRan2}]
We again choose $\theta=1/4$ for notational convenience and do the proof for $z=E+i\eta$ with $\eta>0$ and $E\in\widehat{\Sigma}$. Let $G\subset G' \subseteq \R^d$ and choose$a \in G$ with $Q_a\subset G$ and $b\in G'$. Fix $0<m<M$ with $M:= 2(\|A\|_{\infty}+1)$. Except of the bound for $F_M$ the proof is then the same as the proof of Lemma \ref{lem:CTRan}. We start by rewriting the difference $R_z(A_G)-R_z(A_{G'})$ via the resolvent equation:
\begin{equation}
\chi_G\left( R_z(A_G)-R_z(A_{G'})\right) \chi_{G'} = R_z(A_{G}) \left(  \chi_G A \chi_{G'\setminus G}\right) R_z(A_{G'}).
\end{equation}
H\"older's inequality then yields for $a,b$ as chosen above and $z\in \C$ with $\Im(z)=M$
\begin{align}
&\left\Vert \mathbb{E} \left[  \chi_a \left(R_z(A_G)-R_z(A_{G'})\right) \chi_b  \right]\right\Vert \notag\\
&\qquad \leq \sum_{\substack{k\in (G'\setminus G)_+\\ l\in G_+}} \mathbb{E}\left[\|\chi_a R_z(A_{G'})\chi_k\|^3\right]^{1/3}\mathbb{E}\left[\|\chi_k A\chi_l\|^3\right]^{1/3} \mathbb{E}\left[\|\chi_l R_z(A_{G})\chi_b\|^3\right]^{1/3} \notag\\
&\qquad \leq \frac{C^{1/3}\|A\|_{\infty}^{2/3}}{M^{4/3}}\sum_{\substack{k\in (G'\setminus G)_+\\ l\in G_+}} e^{-\mu|k-l|/3} \mathbb{E}\left[\Vert \chi_aR_z(A_{G'})\chi_k \Vert \right]^{1/3} \mathbb{E}\left[\Vert \chi_l R_z(A_{G})\chi_b \Vert \right]^{1/3},\label{eq:Pf2CtRan1}
\end{align}
where for the last inequality we used \ref{AssLoc2} and estimated $\Vert \chi_a R_z(A_{G})\chi_k \Vert$ respectively $\Vert \chi_k R_z(A_{G})\chi_b \Vert$ by $1/|\Im(z)|=1/M$. The two remaining expectations can now be estimated as in the proof of Lemma \ref{lem:CTRan}. Because the operator kernel of $A_{G^{(\prime)}}$ can be estimated by the operator kernel of $A$, there exist constants $C,\mu>0$, which are independent of $G$ and $G'$, such that
\begin{align}
\mathbb{E}\left[\Vert \chi_aR_{E+iM}(A_{G})\chi_k \Vert \right] &\leq C e^{-\mu|a-k|^{1/4}},\\
\mathbb{E}\left[\Vert \chi_l R_{E+iM}(A_{G'})\chi_b \Vert \right] &\leq C e^{-\mu|l-b|^{1/4}} \label{eq:Pf2CtRan2}
\end{align}
for $a,b\in\R^d$. Estimating \eqref{eq:Pf2CtRan1} via \eqref{eq:Pf2CtRan2} then implies
\begin{align}
F_M &\leq C' \sum_{\substack{k\in (G'\setminus G)_+\\ l\in G_+}} e^{-\mu|a-k|^{1/4}/3}e^{-\mu|k-l|/3}e^{-\mu|l-b|^{1/4}/3}\notag\\
&\leq C'' e^{-\mu' (\dist(a,\partial_{G'}G)^{1/4}+ \dist(b,\partial_{G'}G)^{1/4})}
\end{align}
for constants $C',C'',\mu'>0$. If $b\in G$ with $Q_b\subset G$ and $a \in G'$ the proof follows along the same lines.
\end{proof}

\subsection{Proof of Theorem \ref{lem:AuxRes1} under assumption $(i)$}
\label{sec:HelfSjo}

The following argument is essentially contained in \cite{MR1937430}. 

Via the Helffer-Sj\"ostrand formula we first rewrite the left-hand side of \eqref{lem:AuxRes1Stat} in terms of the resolvents of $A_G$ and $A_{G'}$.
In one of its standard formulations the Helffer-Sj\"ostrand formula states that for a selfadjoint operator $A$ and a compactly supported function $f\in \mathcal{C}_c^n(\R)$, $n\geq 2$, the operator $f(A)$ can be written as
\begin{equation}
\label{eq:HelfSjo1}
f(A) = \frac{1}{2\pi} \int_{\R^2} \mathrm{d}x\mathrm{d}y\, \omega_f(x,y) R_{x+iy}(A),
\end{equation}
where $\omega_f := (\partial_x+i\partial_y) \widetilde{f} $ and $\widetilde{f}$ is a quasi-analytic continuation of $f$, see e.g. \cite{MR1349825}. Moreover, $\widetilde{f}$ can be chosen such that 
\begin{align}
& |\omega_{f}(x,y)|\leq C|y|^{n-1}, \label{eq:HelfSjo11} \\
& \supp(\omega_{f})\subseteq\big(\supp(f)+[-1,1]\big) \times i[-1,1] \label{eq:HelfSjo12},
\end{align}
where the constant $C$ only depends on $f$ and $n$.

Let $h$ be as in Theorem \ref{lem:AuxRes1} and let $n:=\lfloor 2\widetilde{q}+2 \rfloor$. Because $h\in \mathcal{C}^n_c(\R)$ we can choose a quasi-analytic continuation $\widetilde{h}_n$ such that $\omega_{h,n}:=(\partial_x+i\partial_y) \widetilde{h}_n$ meets \eqref{eq:HelfSjo11} and \eqref{eq:HelfSjo12}.
For open subsets $G\subset G' \subseteq \R^d$ and $a,b\in G'$ such that $Q_a\subset G$ the Helffer Sj\"ostrand formula gives   
\begin{align}
\label{eq:PfLemGFB1}
&\chi_a \{h(A_G) - h(A_{G'})\}\chi_b \notag\\ 
&\quad=  \frac{1}{2\pi} \int_{\R^2} \mathrm{d}x\mathrm{d}y\, \omega_{h,n}(x,y) \chi_a\{R_{x+iy}(A_G)-R_{x+iy}(A_{G'})\} \chi_b \notag\\
&\quad =:\frac{1}{2\pi} \int_{\R^2} \mathrm{d}x\mathrm{d}y\, \omega_{h,n}(x,y) T_{x+iy}^{a,b}(G,G'),
\end{align}
where we have abbreviated 
\begin{equation}
T_{x+iy}^{a,b}(G,G') := \chi_a\{R_{x+iy}(A_G)-R_{x+iy}(A_{G'})\} \chi_b.
\end{equation}
Upon averaging both sides of \eqref{eq:PfLemGFB1} we obtain the bound
\begin{equation}
\label{eq:PfLemGFB2}
\begin{aligned}
&\left\Vert\mathbb{E}\left[\chi_a \{h(A_G) - h(A_{G'})\}\chi_b\right]\right\Vert \leq \frac{1}{2\pi} \int_{\R^2} \mathrm{d}x\mathrm{d}y\, \vert\omega_{h,n}(x,y)\vert \big\Vert\mathbb{E}\big[T_{x+iy}^{a,b}(G,G')\big] \big\Vert.
\end{aligned}
\end{equation}
Lemma \ref{lem:CTRan2} implies that for $0<\theta<1/2$ there exist constants $C',\mu>0$ such that 
\begin{equation}
\label{eq:PfLemGFB4}
\begin{aligned}
&\big\Vert\mathbb{E}\big[T_{z}^{a,b}(G,G')\big] \big\Vert \leq \frac{C'}{\distt(z,\widehat{\Sigma})} e^{-\mu\distt(z,\widehat{\Sigma}) \left(\distt(a,\partial_{G'}G)^{\theta}+\distt(b,\partial_{G'}G)^{\theta}\right)}
\end{aligned}
\end{equation}
holds for $z\in\C \setminus \widehat{\Sigma}$.
Estimating the right-hand side of \eqref{eq:PfLemGFB2} by \eqref{eq:PfLemGFB4} yields
\begin{equation}
\label{eq:PfLemGFB5}
\begin{aligned}
&\left\Vert\mathbb{E}\left[\chi_a \{h(A_G) - h(A_{G'})\}\chi_b\right]\right\Vert\\
&\qquad \leq C{''} \int_{-1}^1 \mathrm{d}y\, |y|^{n-2} e^{-\mu|y| \left(\distt(a,\partial_{G'}G)^{\theta}+\distt(b,\partial_{G'}G)^{\theta}\right)},
\end{aligned}
\end{equation}
where we also used \eqref{eq:HelfSjo11} and \eqref{eq:HelfSjo12}.
A change of variables shows that
\begin{equation}
\label{eq:PfLemGFB8}
\eqref{eq:PfLemGFB5} \leq \frac{C'''}{\distt(a,\partial_{G'}G)^{(n-1)\theta}+\distt(b,\partial_{G'}G)^{(n-1)\theta}},
\end{equation}
where the constant $C'''$ depends on $\theta$. Because $n-1 = \lfloor 2\widetilde{q}+2 \rfloor -1 > 2\widetilde{q}$ we can choose $\theta<1/2$ such that $(n-1)\theta > \widetilde{q}$. 

\qed

\section{Proof of Theorem \ref{th:SzegoAsymptAvg}}
\label{sec:ProofTheorem}

The proof of the theorem consists of two parts. In the first part, which is purely algebraic, we rewrite $\mathbb{E}\left[\Tr\left( h(g(H)_{\Lambda_L}) \right)\right]$ via the transformations $\{T_j\}_{j\in\Z^d}$, $\{P_\pi\}_{\pi\in\mathcal{S}^d}$ and $\{R_{\sigma}\}_{\sigma\in\mathcal{R}^d}$ as 
\begin{equation}
\label{eq:ToProveStep1}
\mathbb{E}\big[ \Tr\big( h(g(H)_{\Lambda_L}) \big) \big] = \sum_{m=0}^{d} (2L)^{d-m} A^{(L)}_m  + \mathcal{E}^{(L)},
\end{equation}
where the $A_m^{(L)}$ are finite-volume versions of the coefficients $A_m$ from \eqref{cor:SzegoAsymptAvgstat2} and $\mathcal{E}^{(L)}$ is an error term. In this part of the proof we work with the non-averaged quantities $\Tr\left(h(g(H)_G)\right)$ as long as possible. For a concrete model such as the random Anderson model, and additional (model-specific) assumptions, the pointwise formula \eqref{pf:SzegoP1P2,15} would be the starting point for an almost sure pointwise or stochastic asymptotic analysis beyond the results from \cite{PasturKirschErgodic}. In the second part we then apply Theorem \ref{lem:AuxRes1} to show that the coefficients $A_m$ are well-defined for $\widetilde{q}>2d$ and that there exist constants $C,C'$ such that
\begin{align}
\label{eq:ToProveStep2}
\vert A^{(L)}_m - A_{m} \vert &\leq C L^{2m-\widetilde{q}},\\	
\label{eq:ToProveStep2.2}
\vert \mathcal{E}^{(L)}\vert &\leq C' L^{d-\widetilde{q}}.
\end{align} 
At the end of the section a short calculation verifies formula \eqref{cor:rmk1} in Remark \ref{rem:SzAsym1}.

\subsection{First part of the proof}

We recall the definitions $\{T_j\}_{j\in\Z^d}$, $\{P_\pi\}_{\pi\in\mathcal{S}^d}$ and $\{R_{\sigma}\}_{\sigma\in\mathcal{R}^d}$ from \eqref{def:TheModel2},\ref{assV2} and \ref{assV3}, respectively.
For the whole first part we choose a fixed length scale $L\in \N$.
To shorten notation we use the shortcuts $f_{n}$, $n=0,...,d$, introduced in \eqref{th:SzegoAsymptAvgstat3}. In the same vein we abbreviate
\begin{equation}
\label{eq:ShortCut1}
f^{T}_{n} := h\big(g( H^{T} )_{\mathbb{R}_{\geq 0}^n\times \R^{d-n}}\big)
\end{equation}
for a transformation $T:\Omega\to\Omega$, where $H^T$ is the random operator obtained from $H^{T}_\omega:= H_{T\omega}$, $\omega\in\Omega$. We first decompose the cube $\Lambda_L$ of side length $2L$ into $2^d$ disjoint subcubes
\begin{equation}
\Lambda_{L/2}(\sigma):= \{x\in\R^d:\, x_{\sigma} \in [-L,0]^d\}, \qquad \sigma\in\mathcal{R}^d, 
\end{equation}
of side length $L$, where $x_{\sigma}:= ((-1)^{\sigma_1} x_1,...,(-1)^{\sigma_d} x_d)$. Under the m.p. transformation $R_{\sigma}$ from assumption \ref{assV3} the difference of the operators $h(g(H)_{\Lambda_L})$ and $f_0=(h\circ g)(H)$ transforms as
\begin{align}
U_{\sigma}\chi_{\Lambda_{L/2}(\sigma)} \{ h(g(H)_{\Lambda_L})-f_0 \} U_{\sigma} &=  \chi_{[-L,0]^d}
U_{\sigma}\{ h(g(H)_{\Lambda_L})-f_0 \}U_{\sigma} \notag \\
&= \chi_{[-L,0]^d} \{ h(g(H^{R_{\sigma}})_{\Lambda_L})-f^{R_{\sigma}}_0 \} \label{eq:PfStep1,1}
.
\end{align}
Via the unitary transformations $\{U_{j}\}_{j\in\Z^d}$ we can further rewrite the right-hand side of \eqref{eq:PfStep1,1} as
\begin{align}
&\chi_{[-L,0]^d} \{h(g(H^{R_{\sigma}})_{\Lambda_L})-f^{R_{\sigma}}_0\} \notag\\ 
&\quad= U^*_L\chi_{[0,L]^d} \{h(g( H^{T_L R_\sigma} )_{[0,2L]^d})-f_{0}^{T_LR_\sigma}\}U_L \label{eq:PfStep1,2},
\end{align}
where $U_L$ and $T_L$ are a short-cut for $U_{(L,...,L)}$ and $T_{(L,...,L)}$, respectively. After combining \eqref{eq:PfStep1,1} and \eqref{eq:PfStep1,2} we take the trace and sum over $\sigma\in\mathcal{R}^d$ to arrive at
\begin{equation}
\label{pf:Formula1}
\begin{aligned}
& \Tr\left( \chi_{\Lambda_L} \{ h(g(H)_{\Lambda_L})-f_0 \}  \right)\\  
&\qquad\qquad= \sum_{\sigma\in\mathcal{R}^d} \Tr\left(\chi_{[0,L]^d} \{h(g(H^{T_L R_\sigma})_{[0,2L]^d})-f_{0}^{{T_LR_\sigma}}\} \right). 
\end{aligned}
\end{equation}
With the error term
\begin{equation}
\label{pf:Error}
\mathcal{E}^{(L)} := \sum_{\sigma\in\mathcal{R}^d}  \Tr\left(\chi_{[0,L]^d} \{h(g(H^{T_LR_\sigma})_{[0,2L]^d})-f_{d}^{{T_LR_\sigma}}\} \right)
\end{equation}
the formula \eqref{pf:Formula1} reads 
\begin{equation}
\label{pf:Formula2}
\eqref{pf:Formula1} = \sum_{\sigma\in\mathcal{R}^d} \Tr\left( \chi_{[0,L]^d} \{ f_{d}^{T_L R_\sigma} - f_{0}^{T_L R_\sigma} \} \right)   + \mathcal{E}^{(L)}.
\end{equation}
So far we reduced the problem to a corner of the cube of linear size $L$ and absorbed the effect of those boundary parts of $\Lambda_L$ into an error term that are far apart from the corner under consideration. Let's continue by decomposing the box $[0,L]^d$ as
\begin{equation}
\label{pf:SzegoDecompBox}
[0,L]^d = \bigcup_{\pi\in \mathcal{S}^d} \{ x\in [0,L]^d:\, x_{\pi(1)}\leq ...\leq x_{\pi(d)}\},
\end{equation}
where the union is disjoint up to a set of Lebesgue-measure zero. The single sets on the right-hand side of \eqref{pf:SzegoDecompBox} can be transformed into each other via relabeling coordinates: If we set 
\begin{equation}
\label{pf:SzegoP1P2,1}
\chi_{L,\pi}:=\chi_{[0,L]^d}\chi_{\{x_{\pi(1)}\leq...\leq x_{\pi(d)}\}}
\end{equation}
for $\pi\in \mathcal{S}^d$, then
\begin{equation}
\label{pf:SzegoP1P2,2}
\chi_{L,\pi} = U_{\pi} \chi_{L,\idd} U_{\pi}^*
\end{equation}
(where '$\idd$' here stands for the neutral element in $\mathcal{S}^d$). We extend the shortcut \eqref{eq:ShortCut1} as follows. For a transformation $T:\Omega\to\Omega$ we set 
\begin{equation}
\label{pf:SzegoP1P2,3}
f_{n,\pi}^T := h\big( U_\pi \chi_{(\R_{\geq 0}^n\times \R^{d-n})} U_{\pi}^* g( H^T ) U_\pi \chi_{(\R_{\geq 0}^n\times \R^{d-n})} U_{\pi}^* \big),
\end{equation}
i.e. $f_{n,\idd}^T=f_{n}^T$ as operators on $L^2(\R^d)$. Moreover, $f_{d,\pi}^T=f_{d}^T$ and $f_{0,\pi}^T=f_{0}^T$ hold for any $\pi\in\mathcal{S}^d$. Via a telescopic expansion we arrive at 
\begin{align}
&\Tr\left( \chi_{[0,L]^d} \{ f_{d}^{T_L R_\sigma} - f_{0}^{T_L R_\sigma} \}  \right)  \notag\\
&\qquad = \sum_{\pi\in\mathcal{S}^d} \sum_{n=1}^d \Tr\left( \chi_{L,\pi} \{ f_{n,\pi}^{T_L R_\sigma} - f_{n-1,\pi}^{T_L R_\sigma} \} \right). \label{pf:SzegoP1P2,4}
\end{align}
For $n,l=1,...,d$ we define 
\begin{equation}
\label{pf:SzegoP1P2,5}
\vec{K}^d_{n,l}:= \big\{\vec{k}=(k_1,...,k_{n-1}):\, k_i\in \{1,...,d\}\setminus\{l\},\, k_i\neq k_j (i\neq j)\big\},
\end{equation}
and for $\vec{k}=(k_1,...,k_{n-1})\in \vec{K}^d_{n,l}$
\begin{equation}
\label{pf:SzegoP1P2,6}
\mathcal{S}_n^d(\vec{k},l) := \{\pi\in\mathcal{S}^d:\, (\pi(1),...,\pi(n))=(k_1,...,k_{n-1},l)\} \subseteq \mathcal{S}^d.
\end{equation}
For fixed $n=1,...,d$ the sets $\mathcal{S}_{n}^d(\vec{k},l)$, $l\in \{1,,,.d\}$ and $\vec{k}\in \vec{K}^d_{n,l}$, form a disjoint partition of $\mathcal{S}^d$. Hence \eqref{pf:SzegoP1P2,4} can be written as
\begin{equation}
\label{pf:SzegoP1P2,7}
\eqref{pf:SzegoP1P2,4} = \sum_{n=1}^d\sum_{l=1}^d \sum_{\vec{k} \in \vec{K}^d_{n,l}} \sum_{\pi\in \mathcal{S}^d_n(\vec{k},l)} \Tr\left( \chi_{L,\pi} \{f_{n,\pi}^{T_LR_\sigma}-f_{n-1,\pi}^{T_LR_\sigma}\} \right).
\end{equation}
For fixed $\vec{k},l$ we choose an arbitrary but fixed $\pi_0=\pi_0(\vec{k},l)\in\mathcal{S}^d$ such that $\pi_0^{-1} \in \mathcal{S}_n^d(\vec{k},l)$ and calculate for $\pi \in \mathcal{S}_n^d(\vec{k},l)$
\begin{align}
f_{n,\pi}^{T_L R_\sigma} &= h\big(U_{\pi} \chi_{(\R_{\geq 0}^n\times \R^{d-n})} U_{\pi}^* g(U_{L}U_{\pi_0}^*H^{P_{\pi_0}R_\sigma}U_{\pi_0}U_{L}^*)U_{\pi} \chi_{(\R_{\geq 0}^n\times \R^{d-n})}U_{\pi}^* \big)\notag\\
&= U_{\pi_0}^* h\big( \chi_{(\R_{\geq 0}^n\times \R^{d-n})}  g(U_L H^{P_{\pi_0}R_\sigma}U_L^*) \chi_{(\R_{\geq 0}^n\times \R^{d-n})} \big) U_{\pi_0}\notag\\
&=U_{\pi_0}^* f_{n,\idd}^{T_L P_{\pi_0}R_\sigma} U_{\pi_0}.
\label{pf:SzegoP1P2,8}
\end{align}
Here we used that $U_L$ commutes with $U_{\pi}$, $\pi\in \mathcal{S}^d$, and
\begin{equation}
\label{pf:SzegoP1P2,9}
U_{\pi_0\circ\pi} \chi_{(\R_{\geq 0}^n\times \R^{d-n})} U_{\pi_0\circ\pi}^* = \chi_{(\R^n_{\geq 0}\times \R^{d-n})}.
\end{equation}
Since $(\pi_0\circ \pi)(n) = n$ a similar calculation can be performed for $f_{n-1,\pi}^{T_LR_\sigma}$. Combining them yields
\begin{equation}
\label{pf:SzegoP1P2,10}
\Tr\left( \chi_{L,\pi} \{f_{n,\pi}^{T_LR_\sigma}-f_{n-1,\pi}^{T_LR_\sigma}\} \right)=\Tr\left(  U_{\pi_0}\chi_{L,\pi}U_{\pi_0}^* \{ f_{n}^{T_LP_{\pi_0}R_\sigma}-f_{n-1}^{T_LP_{\pi_0}R_\sigma} \}  \right).
\end{equation}
Via the inclusion-exclusion principle we rewrite the sum over $\pi\in \mathcal{S}_n^d(\vec{k},l)$ of the operators $U_{\pi_0}\chi_{L,\pi}U_{\pi_0}^*$ as
\begin{align}
\sum_{\pi\in \mathcal{S}^d_n(\vec{k},l)}  U_{\pi_0}\chi_{L,\pi}U_{\pi_0}^* &= \chi_{[0,L]^d}\chi_{\{x_1\leq...\leq x_{n}\}} \prod_{i=n+1}^d \chi_{\{x_{n}\leq x_{i} \}} \label{pf:SzegoP1P2,11}\\
&= \chi_{[0,L]^d} \chi_{\{x_1\leq...\leq x_{n}\}} \prod_{i=n+1}^d \left(\chi_{\R^d} - \chi_{\{x_n\geq x_i\}} \right)\notag\\
&= \chi_{[0,L]^d} \chi_{\{x_1\leq...\leq x_{n}\}} \sum_{j=0}^{d-n} (-1)^j \sum_{\substack{\mathcal{M}\subseteq \{n+1,...,d\}:\\ |\mathcal{M}|=j}} \chi_{\{\forall t\in\mathcal{M}:\,x_n\geq x_t\}}.\notag 
\end{align}
For the $j=0 $ summand the second sum is interpreted as $\chi_{\R^d}$. By summing \eqref{pf:SzegoP1P2,10} over $\pi\in \mathcal{S}_n^d(\vec{k},l)$ we obtain
\begin{align}
&\sum_{\pi\in \mathcal{S}^d_n(\vec{k},l)} \Tr\left( \chi_{L,\pi} \{f_{n,\pi}^{T_L R_\sigma}-f_{n-1,\pi}^{T_LR_\sigma}\} \right) \notag\\
&\qquad= \sum_{j=0}^{d-n} (-1)^j \sum_{\substack{\mathcal{M}\subseteq \{n+1,...,d\}:\\ |\mathcal{M}|=j}} \Tr\left(\chi_{L,n,\mathcal{M}} \{ f_{n}^{T_LP_{\pi_0}R_\sigma}-f_{n-1}^{T_LP_{\pi_0}R_\sigma} \}  \right),
\end{align}
 where we abbreviated
\begin{equation}
\label{pf:SzegoP1P2,13}
\chi_{L,n,\mathcal{M}} := \chi_{[0,L]^d}\chi_{\{x_1\leq...\leq x_{n}\}} \chi_{\{\forall t\in\mathcal{M}:\,x_n\geq x_t\}}.
\end{equation}
Let's summarize the above calculation. For $n,m=1,...,d$ and $j=0,...,d-n$ we define
\begin{align}
b_{n,j}^{(L)} &:= (-1)^j \sum_{\sigma\in\mathcal{R}^d}\sum_{l=1}^d \sum_{\vec{k} \in \vec{K}^d_{n,l}} \sum_{\substack{\mathcal{M}\subseteq \{n+1,...,d\}:\\ |\mathcal{M}|=j}} \Tr\left(\chi_{L,n,\mathcal{M}} \{ f_{n}^{T_LP_{\pi_0}R_\sigma}-f_{n-1}^{T_LP_{\pi_0}R_\sigma} \}  \right),\\
b_{m}^{(L)} &:= \sum_{n=1}^m b_{n,m-n}^{(L)}. \label{pf:SzegoP1P2,14}
\end{align}
Our above calculation then shows that
\begin{equation}
\label{pf:SzegoP1P2,15}
\Tr\left( \chi_{\Lambda_L} \{ h(g(H)_{\Lambda_L})-f_0 \} \right) = \sum_{m=1}^d b_m^{(L)} + \mathcal{E}^{(L)}.
\end{equation}
Now we take expectations and exploit that $P_{\pi_0}$ and $R_\sigma$ are m.p. transformations. For $m=1,...,d$, $n=1,...,m$ and $\mathcal{M} \subset \{n+1,...,d\}$ a set of size $|\mathcal{M}|=m-n$ as appearing in the coefficients $b_{n,m-n}^{(L)}$ his yields
\begin{align}
\mathbb{E}\left[\Tr\left(\chi_{L,n,\mathcal{M}} \{ f_{n}^{T_LP_{\pi_0}R_\sigma}-f_{n-1}^{T_LP_{\pi_0}R_\sigma} \}  \right) \right] &= \mathbb{E}\left[\Tr\left(\chi_{L,n,\mathcal{M}} \{ f_{n}-f_{n-1} \}  \right) \right] \label{pf:SzegoP1P2,16}\\
&= \mathbb{E}\left[\Tr\left(\chi_{L,n,\mathcal{M}_{m,n}} \{ f_{n}-f_{n-1} \}  \right) \right]\notag, 
\end{align}
where in the second step we substituted the set $\mathcal{M}$ by $\mathcal{M}_{m,n}:=\{n+1,...,m\}$ (with $\mathcal{M}_{m,m}=\emptyset$). This is possible because the m.p. transformation $T_\mathcal{M}$ associated to the unitary operator $U_\mathcal{M}$, which acts via relabeling the coordinates indexed by $\mathcal{M}$ into those indexed by $\mathcal{M}_{m,n}$, satisfies
\begin{equation}
\label{pf:SzegoP1P2,17}
U_\mathcal{M}\{f_{n}-f_{n-1}\} U_\mathcal{M}^* = f^{T_\mathcal{M}}_{n}-f^{T_\mathcal{M}}_{n-1}.
\end{equation}
The right-hand side of \eqref{pf:SzegoP1P2,16} now is independent of $\sigma\in\mathcal{R}^d$, $l=1,...,d$ and $\vec{k} \in \vec{K}_{n,l}^d$ and therefore
\begin{equation}
\label{pf:SzegoP1P2,18}
\mathbb{E} \big[ b_{n,m-n}^{(L)} \big] =  \frac{(-1)^{m-n} 2^dd!}{(m-n)!(d-m)!} \mathbb{E}\left[\Tr\left(\chi_{L,n,\mathcal{M}_{m,n}} \{ f_{n}-f_{n-1} \}  \right) \right].
\end{equation}
Hence, if we set
\begin{equation}
\label{pf:SzegoP1P2,19}
c_{m,n}:= \frac{(-1)^{m-n} 2^{m}d!}{(m-n)!(d-m)!},
\end{equation}
we arrive at
\begin{equation}
\label{pf:SzegoP1P2,20}
\mathbb{E}\big[b_{m}^{(L)}\big] = 2^{d-m}\sum_{n=1}^m c_{m,n} \mathbb{E}\big[\Tr\left(\chi_{L,n,\mathcal{M}_{m,n}} \{ f_{n}-f_{n-1} \}  \right) \big].
\end{equation}

For $1\leq n \leq m \leq d$ the operator on the right-hand side of \eqref{pf:SzegoP1P2,18}
is invariant under translations in the last $d-m$ coordinates. For a cube $Q_a\subset [0, L]^{d-m}$ of side-length $1$ and with center $a\in (\Z+1/2)^{d-m}$ this gives
\begin{align}
& \mathbb{E}\left[\Tr\left(\chi_{ L,n,\mathcal{M}_{m,n}}\chi_{\{(x_{m+1},...,x_{d})\in Q_a\}} \{ f_{n}-f_{n-1} \}\right)\right]\notag\\
&\quad =  \mathbb{E}\left[\Tr\left(\chi_{ L,n,\mathcal{M}_{m,n}}\chi_{\{(x_{m+1},...,x_{d})\in [0,1]^{d-m}\}} \{ f_{n}-f_{n-1} \}\right)\right].\label{pf:SzegoP1P2,22}
\end{align}
For $m=1,...,d$ we obtain 
\begin{align}
\mathbb{E}\big[b_{m}^{(L)}\big] =&  (2L)^{d-m} \sum_{\substack{n=1}}^m c_{m,n}  \mathbb{E}\left[\Tr\left(\chi_{ L,n,\mathcal{M}_{m,n}}\chi_{\{x_{m+1},...,x_{d}\in [0,1]\}} \{ f_{n}-f_{n-1} \}  \right) \right]\notag\\
=:& (2L)^{d-m} A^{(L)}_m \label{pf:SzegoP1P2,23}
\end{align}
and for $m=0$ we set
\begin{align}
A_0^{(L)} := \mathbb{E}\big[ \Tr\big( \chi_{[0,1]^d} f_{0} \big) \big].
\end{align}
This finishes the first part of the proof, which can be summarized as
\begin{equation}
\mathbb{E}\left[ \Tr\left( h(g(H)_{\Lambda_L}) \right) \right] = \sum_{m=0}^{d} (2L)^{d-m} A_{m}^{(L)}  + \mathbb{E}\big[\mathcal{E}^{(L)}\big].
\end{equation}

\subsection{Second part of the proof}

We start with proving that $\mathcal{E}^{(L)}$ defined in \eqref{pf:Error} is indeed a negligible error term. For a set $U\subseteq\R^d$ we recall the notation $U_+:= \{n\in (\mathbb{Z}+1/2)^d:\, Q_n\cap A \neq \emptyset\}$. Because of \ref{AssAg} and Lemma \ref{lem:AppTrClEst} we may interchange trace and expectation in \eqref{pf:Error} to obtain
\begin{align}
\big|\mathbb{E}\big[\mathcal{E}^{(L)}\big]\big| &=2^d\, \big\vert \Tr\big( \mathbb{E} \big[  \chi_{[0,L]^d} \{ h(g(H)_{[0,2L]^d})-f_{d} \} \big]\big) \big\vert \notag\\
&\leq 2^d\, \sum_{a\in ([0,L]^d)_+} \big\Vert\mathbb{E} \big[  \chi_{a} \{ h(g(H)_{[0,2L]^d})-f_d \} \chi_{a} \big]\big\Vert_1.\label{pf:P2F1}
\end{align}
Next, we apply estimate \eqref{lem:AuxRes1Stat} (which by assumption holds for $\widetilde{q}>2d$) with $G=[0,2L]^d$ and $G'=\mathbb{R}^d_{\geq 0}$, in which case $\distt([0,L]^d,\partial_{\mathbb{R}^d_{\geq 0}}[0,2L]^d) = L$. This implies that
\begin{equation}
\label{pf:P2F3}
\big\Vert \mathbb{E} \big[  \chi_a \{ h(g(H)_{[0,2L]^d})-f_d \} \chi_a   \big]\big\Vert_1 \leq  C L^{-\widetilde{q}}
\end{equation}
holds for $a\in([0,L]^d)_+$, and consequently
\begin{equation}
\label{pf:P2F4}
\vert\mathbb{E}\big[\mathcal{E}^{(L)}\big]\vert \leq C' L^{d-\widetilde{q}}.
\end{equation}
Now let us turn to \eqref{eq:ToProveStep2}. We first introduce the abbreviation 
\begin{equation}
\widehat{\chi}_{L,m,n} := \chi_{ L,n,\mathcal{M}_{m,n}}\chi_{\{x_{m+1},...,x_{d}\in [0,1]\}}
\end{equation}
for $L\in\N\cup\{\infty\}$, and $1\leq n \leq m \leq d$. We note that $\widehat{\chi}_{\infty,m,n} = \widehat{\chi}_{m,n}$, where the latter operators were defined in \eqref{def:const1}.
The natural limiting candidates for the coefficients $A_m^{(L)}$ defined in \eqref{pf:SzegoP1P2,23} are
\begin{align}
\label{def:AkInf1}
A_{m} &:= \sum_{n=1}^m c_{m,n}  \Tr\left(\mathbb{E}\left[\widehat{\chi}_{\infty,m,n} \{ f_{n}-f_{n-1} \} \widehat{\chi}_{\infty,m,n} \right] \right).
\end{align}
Here we exchanged the order of trace and expectation to ensure that the coefficients $A_m$ are well-defined via the bound \eqref{lem:AuxRes1Stat} and the calculation below. To prove convergence of $A_m^{(L)}$ towards $A_{m}$ we prove that the single summands which contribute to $A_m^{(L)}$ converge towards their respective infinite-volume counterparts. For brevity we abbreviate for $1\leq n \leq m \leq d$
\begin{align}
A_{m,n}^{(L)} &:= \Tr\left(\mathbb{E}\left[\widehat{\chi}_{L,m,n} \{ f_{n}-f_{n-1} \}\widehat{\chi}_{L,m,n} \right] \right),\\
A_{m,n} &:= \Tr\left(\mathbb{E}\left[\widehat{\chi}_{\infty,m,n} \{ f_{n}-f_{n-1} \}\widehat{\chi}_{\infty,m,n} \right] \right).
\end{align}
We first prove that that the operator $\mathbb{E}\left[\widehat{\chi}_{\infty,m,n} \{ f_{n}-f_{n-1} \}\widehat{\chi}_{\infty,m,n} \right]$ is trace class. The trace norm of this operator can be estimated via the operator kernel of $f_n-f_{n-1}$ as
\begin{align}
&\left\Vert\mathbb{E}\left[\widehat{\chi}_{\infty,m,n} \{ f_{n}-f_{n-1} \}\widehat{\chi}_{\infty,m,n} \right]\right\Vert_1 \notag\\
&\quad\leq\sum_{\substack{a\in(\R_{\geq 0}^d)_+:\\  \widehat{\chi}_{\infty,m,n}\chi_a\neq 0}}
\hspace{0.5cm}\sum_{\substack{b\in(\R_{\geq 0}^d)_+:\\  \widehat{\chi}_{\infty,m,n}\chi_b\neq 0}} \left\Vert\mathbb{E}\left[ \chi_a \lbrace f_{n}-f_{n-1}\rbrace\chi_{b} \right] \right\Vert_1 \notag\\
&\quad\leq C_1 \Big(\sum_{\substack{a\in(\R_{\geq 0}^d)_+:\\  \widehat{\chi}_{\infty,m,n}\chi_a\neq 0}}\frac{1}{\left(|a_n|+1\right)^{\widetilde{q}/2}}\Big)^2, \label{eq:PfAsymptWell1}
\end{align}
where we used the assumption that \eqref{lem:AuxRes1Stat} holds for $\widetilde{q}$ and that 
\begin{equation}
\partial_{\R_{\geq 0}^{n-1}\times \R^{d-(n-1)}}\big(\R_{\geq 0}^n\times \R^{d-n}\big) = \R_{> 0}^{n-1} \times \{0\}\times \R^{d-n}.
\end{equation}
By definition of the operator $\widehat{\chi}_{\infty,m,n}$ the right-hand side of \eqref{eq:PfAsymptWell1} can be estimated by
\begin{align}
\sqrt{\eqref{eq:PfAsymptWell1}} \leq C_2 \sum_{a_{n}\in\mathbb{N}} \sum_{\substack{a_1,...,a_{n-1}\in\mathbb{N}_0:\\a_i \leq a_{n}+1}} \sum_{\substack{a_{n+1},...,a_{m}\in\mathbb{N}_0:\\ a_i \leq a_{n}+1}} a_n^{-\widetilde{q}/2} \leq C_3 \sum_{a_n\in\mathbb{N}} a_n^{m-1-\widetilde{q}/2},\label{eq:PfAsymptWell2}
\end{align}
which is finite for $\widetilde{q}>2m$. Finally we prove that $|A_m^{(L)}-A_m| = \mathcal{O}(L^{2m-\widetilde{q}})$. We proved above that the operator $T_{m,n}:=\mathbb{E}\left[\widehat{\chi}_{\infty,m,n} \{ f_{n}-f_{n-1} \}\widehat{\chi}_{\infty,m,n} \right]$ is trace class. Cyclicity of the trace then yields
\begin{align}
\vert A_{m,n}^{(L)}-A_{m,n} \vert &= \big\vert  \Tr\big( \chi_{[0,L]^d}T_{m,n} \big) - \Tr\big( T_{m,n} \big) \big\vert \label{pf:FinalPart2}\\
&= \big\vert  \Tr\big( (\chi_{\mathbb{R}^d_{\geq 0}}-\chi_{[0,L]^d}) T_{m,n} (\chi_{\mathbb{R}^d_{\geq 0}}-\chi_{[0,L]^d}) \big) \big\vert \notag\\
&\leq \big\Vert\mathbb{E} \big[  \widehat{\chi}_{\infty,m,n} (\chi_{\mathbb{R}^d_{\geq 0}}-\chi_{[0,L]^d}) \lbrace f_{n}-f_{n-1}\rbrace  (\chi_{\mathbb{R}^d_{\geq 0}}-\chi_{[0,L]^d}) \widehat{\chi}_{\infty,m,n}  \big]\big\Vert_1.\notag 
\end{align}
 With estimates as for \eqref{eq:PfAsymptWell1} and \eqref{eq:PfAsymptWell2} we arrive at
\begin{align}
\eqref{pf:FinalPart2} &\leq C_4  \left(\sum_{a_n = L-1}^\infty a_n^{m-1-\widetilde{q}/2}\right)^2 \leq C_5 L^{2m-\widetilde{q}}. \label{pf:FinalPart8} 
\end{align}
\qed

\subsection{Proof of Remark \ref{rem:SzAsym1}} \label{sec:MainPfPt3}

Because of the analysis in the second part of the proof
\begin{align}
A_{m} = \lim_{\substack{L\to\infty \\ L\in \N}} A_{m}^{(L)}  = \lim_{\substack{L\to\infty \\ L\in \N}} \sum_{n=1}^m c_{m,n}  \mathbb{E}\left[\Tr\left(\widehat{\chi}_{L,m,n} \{ f_{n}-f_{n-1} \} \right)\right] 
\label{eq:Part3E1}
\end{align}
with constants $c_{m,n}$ defined in \eqref{pf:SzegoP1P2,19}.
The presence of the finite-volume projection now allows to rearrange terms in the above sum. This leads to
\begin{align}
A^{(L)}_m =& c_{m,m} \mathbb{E}\left[\Tr\left(\widehat{\chi}_{L,m,m} f_{m}\right) \right] - c_{m,1} \mathbb{E}\left[\Tr\left(\widehat{\chi}_{L,m,1} f_{0} \right)\right]\notag\\
 &+ \sum_{n=1}^{m-1} \big(c_{m,n} \mathbb{E}\left[\Tr\left( \widehat{\chi}_{L,m,n} f_{n} \right)\right] - c_{m,n+1} \mathbb{E}\left[\Tr\left( \widehat{\chi}_{L,m,n+1} f_{n}\big) \right]\right). \label{eq:Part3E2}
\end{align}
We next use that $f_{n}$ is invariant under permutation of the first $n$ and last $d-n$ coordinates to find that for $1\leq n \leq m-1$
\begin{align}
\label{eq:Part3E3}
&\mathbb{E}\left[\Tr\left(\widehat{\chi}_{L,m,n} f_{n} \right)\right] \notag\\
&\quad=\frac{1}{n!} \sum_{k=1}^n \mathbb{E} \left[ \Tr\left( f_n \chi_L \chi_{\{x_{m+1},...,x_{d}\in[0,1]\}}  \chi_{\{x_{k} \geq x_{1},...,x_{m} \}}\right) \right], \\
&\mathbb{E}\left[ \Tr\left(\chi_{L} \chi_{L,m,n+1} f_{n} \right)\right]\notag\\
&\quad= \frac{1}{n!(m-n)} \sum_{k=n+1}^{m} \mathbb{E} \left[\Tr\left( f_n \chi_L \chi_{\{x_{m+1},...,x_{d}\in[0,1]\}}  \chi_{\{x_{k} \geq x_{1},...,x_{m} \}}\right) \right].
\end{align}
For the constants $c_{m,n}$ and $c_{m,n+1}$ the relation
\begin{equation}
\label{eq:Part3E4}
\frac{1}{n!} c_{m,n}  =  \frac{-1}{n!(m-n)} c_{m,n+1} = \frac{d!}{(-1)^{m-n}2^m n!(m-n)!(d-m)!} =: \widetilde{c}_{m,n}
\end{equation}
holds, which yields 
\begin{align}
&c_{m,n} \mathbb{E}\left[\Tr\left( \widehat{\chi}_{L,m,n} f_{n} \right)\right] - c_{m,n+1} \mathbb{E}\left[\Tr\left( \widehat{\chi}_{L,m,n+1} f_{n}\right) \right] \notag\\
&\qquad = \widetilde{c}_{m,n}  \mathbb{E} \left[\Tr\left( f_n \chi_L \chi_{\{x_{m+1},...,x_{d}\in[0,1]\}}\right)  \right].
\label{eq:Part3E5}
\end{align}
After performing similar calculations for the $n=0$ and the $n=d$ term appearing on the right-hand side of \eqref{eq:Part3E2}, we arrive at
\begin{align}
A^{(L)}_m &= \lim_{L\rightarrow\infty} \sum_{n=0}^{m} \widetilde{c}_{m,n} \mathbb{E} \left[\Tr\left( f_n \chi_L \chi_{\{x_{m+1},...,x_{d}\in[0,1]\}}  \right)\right] \label{eq:Part3E6}
\end{align}
\qed

\section*{Acknowledgements}
The author is very grateful to Alexander Sobolev, Bernhard Pfirsch and his PhD advisor Peter M\"uller for many illuminating discussions on this topic.

%\bibliographystyle{mybib}
%\bibliography{bibliography-AOC}

\newcommand{\noopsort}[1]{} \newcommand{\singleletter}[1]{#1}

\end{document}